\documentclass[final, a4paper, 12pt]{article}

\usepackage{amsfonts}
\usepackage{amsmath}
\usepackage{amssymb}
\usepackage{amscd}
\usepackage{amsthm}

\usepackage[mathscr]{euscript}
\usepackage{color}
\usepackage{fancyhdr}
\usepackage{dsfont}
\usepackage[a4paper,scale=0.752]{geometry}
\usepackage{hyperref}

\usepackage{bm}
\usepackage{cite}
\usepackage{showkeys}
\usepackage[english]{babel}
\usepackage{dblaccnt}
\usepackage{accents}
\usepackage{graphicx}

\usepackage{array}

\newtheorem{definition}{Definition}
\newtheorem{lemma}[definition]{Lemma}
\newtheorem{theorem}[definition]{Theorem}

\newtheorem{remark}[definition]{Remark}

\newcommand*{\N}{\ensuremath{\mathbb{N}}}
\newcommand*{\Z}{\ensuremath{\mathbb{Z}}}
\newcommand*{\R}{\ensuremath{\mathbb{R}}}
\newcommand*{\C}{\ensuremath{\mathbb{C}}}
\renewcommand{\i}{\mathrm{i}}
\renewcommand{\phi}{\varphi}
\renewcommand{\d}[1]{\,\mathrm{d}#1 \,}
\newcommand{\dS}{\,\mathrm{dS} \,}

\newcommand{\J}{\mathcal{J}} % Bloch transform

\renewcommand{\Re}{\mathrm{Re}\,}
\renewcommand{\Im}{\mathrm{Im}\,}

\newcommand{\W}{{W_{\hspace*{-1pt}\Lambda}}} % Wigner-Seitz cell
\newcommand{\Wast}{{W_{\hspace*{-1pt}\Lambda^*}}} % Brillouin zone
\newcommand{\p}{{\mathrm{p}}} % \bm{0}
\renewcommand{\rho}{{\varrho}} 
\renewcommand{\epsilon}{{\varepsilon}}
\newcommand{\loc}{{\mathrm{loc}}} 

\newlength{\dhatheight}
\newcommand{\dhat}[1]{%
    \settoheight{\dhatheight}{\ensuremath{\hat{#1}}}%
    \addtolength{\dhatheight}{-0.35ex}%
    \hat{\vphantom{\rule{1pt}{\dhatheight}}%
    \smash{\hat{#1}}}}

\usepackage{color}

\begin{document}

\sloppy

\title{A Convergent Numerical Scheme for Scattering\\ of Aperiodic Waves from Periodic Surfaces\\ Based on the Floquet-Bloch Transform}
\author{Armin Lechleiter\thanks{Center for Industrial Mathematics, University of Bremen%, Bremen, Germany
; \texttt{lechleiter@math.uni-bremen.de}} \and 
Ruming Zhang\thanks{Center for Industrial Mathematics, University of Bremen%, Bremen, Germany
; \texttt{rzhang@uni-bremen.de}; corresponding author}}
\date{}
\maketitle

\begin{abstract}
Periodic surface structures are nowadays standard building blocks of optical devices. 
If such structures are illuminated by aperiodic time-harmonic incident waves as, e.g., Gaussian beams, the resulting surface scattering problem must be formulated in an unbounded layer including the periodic surface structure. 
An obvious recipe to avoid the need to discretize this problem in an unbounded domain is to set up an equivalent system of quasiperiodic scattering problems in a single (bounded) periodicity cell via the Floquet-Bloch transform.
The solution to the original surface scattering problem then equals the inverse Floquet-Bloch transform applied to the family of solutions to the quasiperiodic problems, which simply requires to integrate these solutions in the quasiperiodicity parameter.  
A numerical scheme derived from this representation hence completely avoids the need to tackle differential equations on unbounded domains. 
In this paper, we provide rigorous convergence analysis and error bounds for such a scheme when applied to a two-dimensional model problem, relying upon a quadrature-based approximation to the inverse Floquet-Bloch transform and finite element approximations to quasiperiodic scattering problems.  
Our analysis essentially relies upon regularity results for the family of solutions to the quasiperiodic scattering problems in suitable mixed Sobolev spaces. 
We illustrate our error bounds as well as efficiency of the numerical scheme via several numerical examples. 
\end{abstract}
%
% Old abstract: 
%
% We use the (Floquet-)Bloch transform to reformulate scattering problems for non-periodic incident fields from an unbounded periodic surface as a family of quasiperiodic scattering problems formulated on a bounded domain. 
% A numerical approximation of the inverse Bloch transform then allows to approximate the solution to the original scattering from the unbounded surface via numerical solutions to several of the quasiperiodic problems.
% The resulting numerical scheme merely needs to tackle quasiperiodic scattering problems on a bounded domain, avoiding the need to approximate solutions on an unbounded periodic domain. 
% 

\section{Introduction}
Time-harmonic scattering from unbounded periodic surface structures is a well-established topic in applied mathematics if one merely considers periodic or quasiperiodic incident fields such as, e.g., incident plane waves or incident periodic point sources.
Under this assumption one can reduce the scattering problem to a single unit cell of the periodic structure such that it usually suffices to apply standard techniques for existence and approximation theory, see, e.g.,~\cite{Chen1991, Nedel1991, Dobso1992a, Abbou1993, Bonne1994}. 
If aperiodic incident fields such as Gaussian beams or multi-poles with a single source point illuminate the periodic surface, this reduction fails such that the resulting surface scattering problem is naturally formulated in the unbounded domain above the surface; variational formulations can then be set up using exterior Dirichlet-to-Neumann operators in unbounded horizontal layers of finite height, see~\cite{Chand2005, Chand2010} and Figure~\ref{fig:0}.

Decomposing the incident field into its quasiperiodic components however shows that the total wave field actually equals the inverse (Floquet-)Bloch transform applied to a family of solutions to quasiperiodic scattering problems with right-hand side equal to the Bloch transform of the aperiodic incident field. 
This trick hence allows to completely avoid the need to deal with a scattering problem formulated on an unbounded horizontal strip above the periodic surface and potentially allows codes for quasiperiodic problems to tackle scattering of aperiodic incident fields.  
On the downside, one needs to be able to (analytically or numerically) compute the Bloch transform of the incident wave; semi-analytic expressions are for instance available for incident point sources or Herglotz wave functions, which are approximate for instance Gaussian beams, see~\cite{Lechl2016, Lechl2015e}. 
\begin{figure}\label{fig:0}
\begin{center}
    \includegraphics[width=0.6\linewidth]{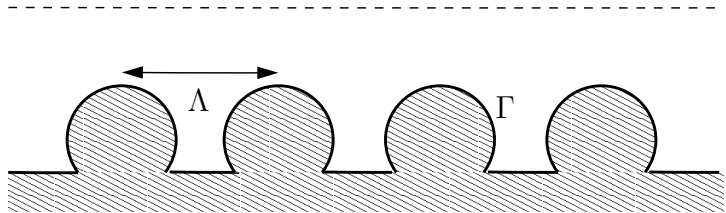}
  \end{center}
  \caption{Sketch of a periodic surface structure in two dimensions defined by a periodic surface $\Gamma$ wih period $\Lambda$. 
  Scattering problems for aperiodic incident fields are naturally formulated in the domain above $\Gamma$ and their variational formulations, relying on an exterior Dirichlet-to-Neumann operators, are typically set up in some layer of finite height in between $\Gamma$ and the dashed line.}
\begin{picture}(0,0)%
 \put(5.55,5.2){$\Lambda$}
 \put(9.6,5.1){$\Gamma$}
\end{picture}
\end{figure}%

Despite the Bloch transform seems to be part of the folklore of applied analysis, we are unaware of any algorithm in scattering theory--let alone convergence analysis--based on this representation.
In this paper, we hence aim to provide such an analysis together with associated error bounds for a numerical scheme relying on the Bloch transform to tackle surface scattering of aperiodic fields from periodic structures. 
To this end, we choose time-harmonic scattering described by the Helmholtz equation from a periodic surface with Dirichlet boundary condition in two dimensions as a model problem. 
Whilst the Dirichlet boundary condition might be replaced by, e.g., impedance or Robin-type conditions, the two-dimensional setting is somewhat crucial for our analysis: 
Central to our convergence result is a regularity theorem stating that under suitable assumptions the family of solutions to the quasiperiodic scattering problems is weakly differentiable in the quasiperiodicity parameter and the weak derivative belongs to $L^q$ for $1 \leq q < 2$ (but generally not for $q=2$). 
(This regularity theorem is based on several auxiliary results on the Bloch transform from~\cite{Lechl2016} that we state without proof.)
In particular, in two dimensions this family depends continuously on the quasiperiodicity such that standard interpolation projections are attractive for the  numerical approximation of the inverse Bloch transform. 
In higher dimensions, these interpolation projections might have to be replaced by more involved projection operators of Cl\'ement-type.  

Let us now briefly sketch the surface scattering problem we consider, together with the numerical scheme we propose to approximate its solution. 
Suppose that $\Omega\subset \R^2$ is a domain with $\Lambda$-periodic boundary $\Gamma$, see Figure~\ref{fig:0}, that is illuminated by some incident field $u^i$, which is a twice continuously differentiable solution to the Helmholtz equation $\Delta u^i + k^2 u^i = 0$ in $\overline{\Omega}$. 
(Strictly speaking, we merely have to assume that $u^i$ is smooth apart from, possibly, a lower-dimensional manifold where the incident field is generated.) 
The total field $u: \, \Omega \to \C$ then satisfies the Helmholtz equation $\Delta u + k^2 u = 0$ in $\Omega$, subject to Dirichlet boundary conditions $u=0$ on $\partial\Omega$. 
Finally, the scattered field $u^s = u - u^i$ has to satisfy a radiation conditions that we introduce below in~\eqref{eq:URC}. 
This model for instance describes electromagnetic scattering in TE mode from periodic surfaces independent of one spatial variable, see~\cite{Bonne1994}. 
The Bloch transform $w= \J_\Omega u$ of $u$ is defined by 
\[
  w(\alpha,x)  = \J_\Omega u(\alpha,x) :=  \left[\frac{\Lambda}{2\pi}\right]^{1/2} \sum_{j\in \Z} u(x_1 + \Lambda j, x_2) \, e^{-\i \, \Lambda j \, \alpha}, 
  \quad 
  \alpha \in \R, \ x= \left( \begin{smallmatrix} x_1 \\ x_2 \end{smallmatrix} \right) \in \Omega.
\]
This function is $\alpha$-quasiperiodic in its second argument, that is, $\J_\Omega u(\alpha,x+\Lambda) = \exp(\i\Lambda\alpha) \, \J_\Omega u(\alpha,x)$ holds for all $x\in\Omega$, such that its first argument $\alpha$ is called the quasiperiodicity. 
Further, $\J_\Omega u$ is $\Lambda^* = 2\pi/\Lambda$-periodic in $\alpha$ and $w(\alpha,\cdot)$ solves an $\alpha$-quasiperiodic scattering problem for the quasiperiodic incident field $\J_\Omega u^i$, see Section~\ref{se:qpScatt}. 
Solving this quasiperiodic problem for $N$ quasiperiodicities $\alpha_j$, $j=1,\dots,N$, by some convergent approximation scheme (we will rely on finite elements) hence yields discrete solutions $w_h(\alpha_j,\cdot)$. 
For points in $\Omega_H^\Lambda$, the inverse Bloch transform $\J_\Omega^{-1} w$ equals a constant $c_\Lambda$ times  $\int_{-\pi/\Lambda^\ast}^{\pi/\Lambda^*} w(\alpha, \cdot) \d{\alpha}$, such that we can approximate the exact solution $u = \J_\Omega^{-1} w$ in $\Omega_H^\Lambda$ by applying, e.g., the trapezoidal rule to the latter integral, $u_{N,h}(x) = 2\pi c_\Lambda/(N \Lambda^\ast) \sum_{j=1}^N w_h(\alpha_j, x)$ for $x \in \Omega_H^\Lambda$.
Under suitable assumptions on the incident field $u^i$, our main result (see Theorem~\ref{th:main}) shows that $\| u_{N,h} - u \|$ is bounded in the $L^2$-norm (or the $H^1$-norm) by some constant times $h^{2} + N^{-1}$ (or $h + N^{-1}$). 

The rest of this paper is structured as follows: 
The next Section~\ref{se:scatter} introduces the surface scattering problem from a periodic surface for basically arbitrary incident fields. 
Section~\ref{se:bloch} then introduces the Bloch transform on the periodic domain $\Omega$. 
This transform allows to reduce the surface scattering problem from Section~\ref{se:scatter} to a family of quasiperiodic scattering problems, see Section~\ref{se:qpScatt}. 
The inverse Bloch transform applied to an associated family of solutions then allows straightforward discretization by quadrature, which is analyzed in Section~\ref{se:errInvBloch}. 
Together with error estimates for finite element approximations to these quasiperiodic solutions we prove in Section~\ref{eq:errFemTotal} convergence of the resulting discrete approximation to the original surface scattering problem. 
Section~\ref{se:num} contains numerical examples confirming the theoretic convergence rates. 
 
\textit{Notation:} 
We write $x = (x_1,x_2)^\top$ or $y=(y_1,y_2)^\top$ for points in $\R^2$. 
The space of smooth functions in a domain $U$ with smooth extension up to arbitrarily high order to the boundary is $C^\infty(\overline{U})$. 
Constants $C$ and $c$ are generic and might change from line to line, and $\nu$ generically denotes the exterior unit normal field to a domain. 
% We further use the symbol $\simeq$ to indicate equivalence of two expressions up to positive constants.

\section{Aperiodic Incident Waves and Periodic Surfaces}\label{se:scatter}
We consider wave scattering from a $\Lambda$-periodic surface $\Gamma = \{ \zeta(y_1,0): \, y_1\in \R \}$ in the periodic domain of propagation $\Omega= \big\{ \zeta(y): \, y\in\R^2, \, y_2>0  \big\}$, both defined by a $\Lambda$-periodic Lipschitz diffeomorphism $\zeta: \, \R^2 \to \R^2$, i.e., $\zeta(x_1+\Lambda,x_2) = \zeta(x_1,x_2)$ for all $x\in\R^2$. 
We always assume that $\zeta = (\zeta(y)_1,\zeta(y)_2)^\top$ satisfies $\zeta(y)_2 > 0$ if $y_2\geq 0$, such that $\overline\Omega \subset \{ y\in \R^2: \, y_2 > 0\}$. 
Let us further fix some $H_0>0$ such that $\Gamma \subset \{ x_2 < H_0 \}$, set $\Omega_H = \{ x \in \Omega: \, x_2 < H \}$ for $H \geq H_0$. 
% , and define $\nu$ as the exterior (downwards pointing) unit normal to $\Omega$ on $\Gamma$.
As mentioned in the introduction, we rely on the Helmholtz equation at wave number $k>0$ for a scalar function $u$, 
\begin{equation}\label{eq:HE}
  \Delta u + k^2 u = 0 \quad \text{in } \Omega \subset \R^2 
\end{equation}
as a model for time-harmonic wave propagation. 
The so-called total field $u$ is caused by scattering of an incident wave $u^i$ that solves the Helmholtz equation from~\eqref{eq:HE} in a neighborhood of $\Gamma$, due to the Dirichlet boundary condition for the total field 
\begin{equation}\label{eq:BC}
  u = 0 \quad \text{on } \Gamma = \partial \Omega,
\end{equation}
and a radiation condition for the scattered field $u^s = u-u^i$ that arises as $u^i$ typically fails to satisfy~\eqref{eq:BC}. 
For variational solutions to scattering problems involving unbounded surfaces, one typically considers the so-called angular spectrum representation as radiation condition, see~\cite{Chand2005}. 
Setting $\sqrt{k^2 - |\xi|^2}  = \i \sqrt{|\xi|^2 - k^2}$ in case that $|\xi|^2 > k^2$, we hence require that the scattered field $u^s$ can be represented in the half space $\{ x_2>H_0 \}$ as  
\begin{equation}\label{eq:URC}
  u^s(x) = \frac{1}{2\pi} \int_{\R} e^{\i x_1  \xi + \i \sqrt{k^2 - |\xi|^2} (x_2-H_0)} \dhat{u}^s(\xi,H_0) \d{\xi} 
  \quad \text{for } x_2 > H_0.
\end{equation}
Here, $\dhat{u}^s(\xi,H_0)$ denotes the Fourier transform of $\left. u^s \right|_{\{ x_2 = H_0\}}$, defined by 
\begin{equation}
  \label{eq:fourierTrafo}
  \dhat{\phi}(z) := \mathcal{F}\phi(z) = \frac{1}{(2\pi)^{1/2}} \int_{\R} e^{-\i  z  x_1} \phi(x_1) \d{x_1} 
  \quad \text{for } \phi \in \C^\infty_0(\R, \C) \text{ and } z \in \R,
\end{equation}
and extended by density to functions in $L^2(\R)$. 
(For simplicity, we identify functions defined the hyperplane $\Gamma_H$ with functions on the real line.) 
If $u$ satisfies the radiation condition ~\eqref{eq:URC}, then the representation in~\eqref{eq:URC} actually holds true when $H_0$ is replaced by any $H>H_0$. 

We now set up a variational formulation for a weak solution $u$ to the scattering problem. 
For incident fields $u^i$ that belong to   
\[
  H^1_r (\Omega_H) = \left\{ u \in \mathcal{D}'(\Omega_H): \, (1+|x_1|^2)^{r/2} u \in H^1(\Omega_H) \right\} 
  \quad \text{for some } H \geq H_0  
\] 
and some $r\in\R$, we seek $u \in H^1_r (\Omega_H)$ that satisfies~\eqref{eq:BC} in the trace sense and the radiation condition~\eqref{eq:URC} for $x_2 > H_0$.  
To this end, we note that restriction of the equality in~\eqref{eq:URC} to $\Gamma_H$ provides a link between the normal derivative of $u^s$ on $\Gamma_H$ and the exterior Dirichlet-to-Neumann operator $T^+$, 
\begin{equation}\label{eq:T}
  \frac{\partial u^s}{\partial x_2}(x_1 , H) 
  = \frac{\i}{(2\pi)^{1/2}} \int_{\R} \sqrt{k^2 - |\xi|^2} \, e^{\i x_1 \xi} \, \dhat{u}^s(\xi,H) \d{\xi}
  =: T^+ \left( u|_{\Gamma_H} \right)(x_1 , H). 
\end{equation}%\marginpar{\high{Check continuity for $H^{1/2}_r$}}%
The operator $T^+$ is continuous from $H^{1/2}_r(\Gamma_H)$ into $H^{-1/2}_r(\Gamma_H)$ for all $|r|<1$, see~\cite{Chand2010, Chand2005}.
Integrating the Helmholtz equation~\eqref{eq:HE} against 
\[
  v \in \widetilde{H}^1_r (\Omega_H) = \{ v \in H^1_r (\Omega_H): \,  v|_{\Gamma} = 0 \} 
\] 
and integrating by parts thus shows that 
\begin{align*}
  0 & = \int_{\Omega_H} \left[ - \Delta u \overline{v} - k^2 u \,\overline{v} \right] \d{x}
  = \int_{\Omega_H} \left[ \nabla u \cdot \nabla \overline{v} - k^2 u \,\overline{v} \right] \d{x}
  - \int_{\Gamma_H} \frac{\partial u}{\partial \nu } \overline{v} \dS \\
  & = \int_{\Omega_H} \left[ \nabla u \cdot \nabla \overline{v} - k^2 u \,\overline{v} \right] \d{x}
  - \int_{\Gamma_H} \left[ \frac{\partial u^i}{\partial \nu } + T^+ [ u-u^i ] \big|_{\Gamma_H}  \right] \overline{v} \dS,
\end{align*}
where we exploited that $\partial u^s / \partial \nu = \partial u^s / \partial x_2 = T^+(\left. [ u-u^i ] \right|_{\Gamma_H})$ on $\Gamma_H$. 
The variational formulation of~\eqref{eq:HE}-\eqref{eq:BC} together with the radiation condition~\eqref{eq:URC} is hence to find $u \in \widetilde{H}^1_r (\Omega_H)$ such that 
\begin{equation}\label{eq:varFormHEScal}
  \int_{\Omega_H} \left[ \nabla u \cdot \nabla \overline{v} - k^2 u \,\overline{v} \right] \d{x}
  - \int_{\Gamma_H}  T^+ [ u ]\big|_{\Gamma_H} \overline{v} \dS 
  = \int_{\Gamma_H} \left[ \frac{\partial u^i}{\partial x_2} - T^+ [u^i]\big|_{\Gamma_H} \right] \overline{v} \dS   
\end{equation}
for all $v \in \widetilde{H}^1_r (\Omega_H)$ with compact support in $\overline{\Omega_H}$.
(Choosing test functions with compact support makes the variational formulation well-defined, independent of whether continuity of $T^+$ from $H^{1/2}_r(\Gamma_H)$ into $H^{-1/2}_r(\Gamma_H)$ holds or not.)
% (Thus, the sesquilinear form in~\eqref{eq:varFormHEScal} is not necessarily bounded for $|r|>1$.) 
We are going to tackle this surface scattering problem using the Bloch transform, introduced in the next section. 
\begin{remark}
  If $\Gamma = \partial\Omega = \{ x\in \R^2: \, x_2 > \zeta(x_1) \}$ additionally is graph of a piecewise continuous function $\zeta: \, \R \to \R$, then it is known that~\eqref{eq:varFormHEScal} is uniquely solvable for all $|r|<1$ if the incident field $u^i$ defines a right-hand side in the dual of $\widetilde{H}^1_r (\Omega_H)$, see~\cite{Chand2010}. 
\end{remark}

\section{The Bloch Transform}\label{se:bloch}
The Bloch transform reduces acoustic scattering problems from periodic surfaces with non-periodic boundary data to quasiperiodic scattering problems from the unit cell of the periodic surface. 
We first define a one-dimensional Bloch transform $\J_{\R}$ on smooth functions $\phi$ with compact support by  
\begin{equation}
  \label{eq:BlochZeta}
  \J_{\R} \phi(\alpha, x_1) 
  := \left[\frac{\Lambda}{2\pi}\right]^{1/2} \sum_{j\in \Z} \phi(x_1 + \Lambda j) \, e^{-\i \, \Lambda j \, \alpha},
  \quad  \alpha \in \R, \, x_1 \in \R. 
\end{equation}
A standard reference for this transform is~\cite{Kuchm1993}. 
Further,~\cite[Annexe B]{Fliss2009} is an excellent source for properties of the one-dimensional Bloch transform.

One easily computes that $\J_\R \phi(\alpha, \cdot)$ is $\alpha$-quasiperiodic for the period~$\Lambda$, 
\begin{equation} \label{eq:qpExt}
  \J_{\R} \phi(\alpha, x_1 +\Lambda)
  = \left[\frac{\Lambda}{2\pi}\right]^{1/2} \sum_{j\in \Z} \phi(x_1 + \Lambda (j+1)) \, e^{-\i \, \Lambda j \, \alpha}
  %= \left[\frac{\Lambda}{2\pi}\right]^{1/2} \sum_{j\in \Z} \phi(x_1 + \Lambda j) \, e^{\i \alpha \cdot \Lambda (j-1)}
  = e^{\i \, \Lambda j \, \alpha} \J_{\R} \phi(\alpha, x_1), 
  \quad  x_1 \in \R.
\end{equation}
Further, $\J_\R \phi$ is for fixed $x_1$ a Fourier series in $\alpha$ with basis functions $\alpha \mapsto \exp(-\i \, \Lambda j \, \alpha)$ that are $\Lambda^* = 2\pi/\Lambda$ periodic. 
Thus, setting 
\[
  \W = \bigg(-\frac\pi\Lambda, \frac\pi\Lambda \bigg]
  \quad \text{and} \quad 
  \Wast = \bigg( -\frac\pi{\Lambda^*}, \frac\pi{\Lambda^*}\bigg] 
  = \bigg( -\frac{\Lambda}{2}, \frac{\Lambda}{2}\bigg] 
\] 
show that knowledge of $(\alpha,x_1) \mapsto \J_{\R} \phi(\alpha, x_1)$ in $\Wast \times \W$ defines that function everywhere in $\R \times \R$.  
%
% Note that the multiplication of $\J_\R$ by $\exp(\i \, x_1 \, \alpha)$ interchanges these two properties, as $\exp(\i \, \alpha \, x_1) \J_\R \phi(\alpha, x_1) $ is $\Lambda$-periodic in $x_1$ and $x_1$-quasiperiodic in~$\alpha$. 
%
%Let us further note already here that the Bloch transform commutes with $\Lambda$-periodic functions: If $q: \, \R \to \C$ is $\Lambda$-periodic, then 
%\begin{equation}\label{eq:JComPerio}
%  [\J_{\R} (q\phi)](\alpha, x_1) 
%  = \frac{|\det \Lambda|}{2\pi}^{1/2} \sum_{j\in\Z} q(x_1 + \Lambda j) \phi(x_1 + \Lambda j) e^{\i\alpha\cdot \Lambda j}
%  = q(x_1) \, (\J_{\R} \phi) (\alpha, x_1).
%\end{equation}
%
To indicate mapping properties of $\J_{\R}$ we recall the Bessel potential spaces $H^s(\R)$ for $s\in \R$, together with their weighted analogues, 
\[
  H^s_r(\R) := \left\{ \phi \in \mathcal{D}'(\R): \, x_1 \mapsto (1+|x_1|^2)^{r/2} \phi(x_1) \in H^s(\R) \right\},
  \qquad s,r \in \R, 
\]
equipped with the norm $\| \phi \|_{H^s_r(\R)} = \| x_1 \mapsto (1+|x_1|^2)^{r/2} \phi(x_1)  \|_{H^s(\R)}$.
To define spaces of periodic functions, note that the smooth, $\Lambda$-periodic functions 
\begin{equation}\label{eq:LPerF}
  \phi_\Lambda^{(j)}(x_1) := \Lambda^{-1/2} \ e^{\i \, \Lambda^* j\,  x_1}, \qquad j\in\Z,
\end{equation}
form a complete orthonormal system in $L^2(\W)$.
For any $\alpha \in \Wast$, the space $\mathcal{D}_\alpha '(\R, \C)$ contains all $\alpha$-quasiperiodic distributions $\phi$ with respect to $\Lambda$, i.e., the products of all periodic distributions, see~\cite{Saran2002}, with $x_1 \mapsto \exp(\i \alpha \,  x_1)$. 
For such distributions and $j \in \Z$, we define Fourier coefficients  
\begin{equation}
  \label{eq:fourierCoeff}
  \hat{\phi}(j) := \phi\left(x_1 \mapsto \overline{\exp(\i \alpha \, x_1) \phi_\Lambda^{(j)}(x_1)}\right) 
  \, \left[ = \frac1{\sqrt\Lambda} \int_{\W} \phi(x_1) e^{-\i [\Lambda^* j + \alpha] x_1 } \d{x_1}
  \ \text{if } \phi \in L^2(\W) \right].
\end{equation}
Further, for $s\in\R$ we introduce the subspace $H^s_\alpha(\W)$ of $\mathcal{D}'_\alpha(\R, \C)$ containing all $\alpha$-quasiperiodic distributions $\phi$ with finite norm $\| \phi \|_{H^s_\alpha(\W)} = ( \sum_{j \in \Z} (1+|j|^2)^s \, |\hat{\phi}(j)|^2 )^{1/2}$. 
Elements of $H^s_\alpha(\W)$ can be represented by their Fourier series, i.e.,  
\[
  \phi(x_1) = \sum_{j\in\Z} \hat{\phi}(j) e^{\i \alpha \, x_1 } \phi_\Lambda^{(j)}(x_1) 
  = \frac{1}{| \det \Lambda|^{1/2}} \sum_{j\in\Z} \hat{\phi}(j) e^{\i (\Lambda^* j+\alpha) \, x_1 }   
  \quad \text{holds in $H^s_\alpha(\W)$.}
\]
We have already noted above that the Bloch transform $\J_{\R}\phi$ extends to a $\Lambda^*$-periodic function in $\alpha$ and to a quasiperiodic function in $x_1$ with quasiperiodicity $\alpha$. 
It is natural that we require adapted function spaces in $(\alpha,x_1)$.  
To this end, we introduce the vector space $\mathcal{D}_{\Lambda}'(\R^2)$ of distributions in $\mathcal{D}'(\R \times \R)$ that are $\Lambda^*$-periodic in their first and quasiperiodic with respect to $\Lambda$ in their second variable, with quasiperiodicity equal to the first variable.
For integers $r\in \N$ and $s\in \R$, these distributions allow to define a norm via 
\begin{equation}\label{eq:HrHs}
  \| \psi \|_{H^\ell_\p(\Wast; H^s_\alpha(W))}^2
  =  \sum_{\gamma = 1}^\ell \int_\Wast \| \partial^{\gamma}_\alpha \psi(\alpha, \cdot) \|_{H^s_\alpha(\W)}^2 \d{\alpha}, 
\end{equation}
and the corresponding Hilbert space as $H^r_\p(\Wast; H^s_\alpha(W)) = \{ \psi \in \mathcal{D}_{\Lambda}'(\R^2):\, \| \psi \|_{H^\ell_\p(\Wast; H^s_\alpha(W))} < \infty \}$. 
Interpolation in $\ell$ and a duality argument subsequently allows to define these spaces for all $r \in \R$, see~\cite{Lions1972, Bergh1976, Fliss2009} or~\cite[Th.~4]{Lechl2016}.

\begin{theorem}\label{th:BlochR}
(a) The Bloch transform $\J_{\R}$ extends from $C^\infty_0(\R)$ to an isometric isomorphism between $L^2(\R)$ and $L^2(\Wast; L^2(\W))$ with inverse
\begin{equation}  \label{eq:X1}
    \left( \J_{\R}^{-1} \tilde{\phi} \right) (x_1 + \Lambda j)
    = \left[\frac{\Lambda}{2\pi}\right]^{1/2} \int_{\Wast} \tilde{\phi}(\alpha, x_1) e^{\i \, \Lambda j \alpha} \d{\alpha}, 
    \qquad \text{for } x_1 \in \W, \, j\in\Z.  
\end{equation}
Further, the $L^2$-adjoint of $J_\R^*: \, L^2(\Wast; L^2(\W)) \to L^2(\R)$ equals its inverse.\\[1mm]
(b) For $s$ and $r \in \R$, the Bloch transform $\J_{\R}$ extends from $C^\infty_0(\R)$ to an isomorphism between $H^s_r(\R)$ and $H^r_\p(\Wast; H^s_\alpha(\W))$. 
Its inverse transform is given by~\eqref{eq:X1} with equality in the sense of the norm of $H^s_r(\R)$. 
\end{theorem}

Next we define an analogous Bloch transform between Sobolev spaces on periodic domains. 
We have already introduced weighted Sobolev spaces $H^s_r(\Omega_H)$ on the unbounded domain $\Omega_H$ in Section~\ref{se:scatter}. 
As a further ingredient for the subsequent result on a volumetric Bloch transform $\J_\Omega$, we consider the set of smooth functions with compact support in $\Omega$, restrict these functions to $\Omega_H$, and denote their closure in $H^s_r(\Omega_H)$ by $\widetilde{H}^s_r(\Omega_H)$. 
(For $s=1$, we already used such spaces in Section~\ref{se:scatter}.)
To state mapping properties of 
\begin{equation} \label{eq:blochOmega}
  \J_\Omega u (\alpha,x) 
  = \left[ \frac{\Lambda}{2\pi} \right]^{1/2} \sum_{j\in \Z} u\left( \begin{smallmatrix} x_1 +\Lambda j \\ x_2 \end{smallmatrix} \right) e^{-\i \, \Lambda j \, \alpha}, 
  \qquad  x = \left( \begin{smallmatrix} x_1 \\ x_2 \end{smallmatrix} \right) \in \Omega_H, \, \alpha \in \R,
\end{equation}
defined for $u \in C^\infty(\overline\Omega_H)$ with compact support, let us further introduce suitable quasiperiodic spaces on the restriction $\Omega_H^\Lambda  = \big\{ x\in\Omega_H: \, x_1 \in \W \big\}$ of $\Omega_H$ to the fundamental domain of periodicity $\W$.
%Thus, $\Omega_H = \bigcup \{ (x_1 +j,x_2)^\top: \, (x_1 ,x_2)^\top \in \Omega_H^\Lambda, \, j \in \Z \}$ and $\Gamma = \bigcup \{ (x_1 +j,x_2)^\top: \, (x_1 ,x_2)^\top \in \Gamma_\Lambda, \, j \in \Z \}$. 
For $\alpha \in \Wast = (-\Lambda/2, \Lambda/2]$, we introduce the Hilbert space 
\begin{equation}
  H^s_\alpha(\Omega_H^\Lambda)
  = \left\{ u = U|_{\Omega_H^\Lambda}: \, U \in H^s_{-1}(\Omega_H) \text{ is $\alpha$-quasiperiodic with respect to } \Lambda \right\},
\end{equation}
with obvious norm and inner product. 
(The decay of $U$ is arbitrary as long as the decay parameter $r$ of the space $H^s_{r}$ is less than $-1/2$.)
Again, $\widetilde{H}^s_\alpha(\Omega_H^\Lambda)$ is the closure of smooth, $\alpha$-quasiperiodic functions on $\Omega$ with compact support in the norm of $H^s_\alpha(\Omega_H^\Lambda)$. 

\begin{theorem}[Th.~8 in~\cite{Lechl2016}]\label{th:BlochOmega}
% (a)  
The Bloch transform $\J_\Omega$ extends to an isomorphism between $H^s_r(\Omega_H)$ and $H^r_\p(\Wast; H^s_\alpha(\Omega_H^\Lambda))$ as well as between $\widetilde{H}^s_r(\Omega_H)$ and $H^r_\p(\Wast; \widetilde{H}^s_\alpha(\Omega_H^\Lambda))$ for all $s,r \in \R$. 
Further, $\J_\Omega$ is an isometry for $s=r=0$ with inverse 
\begin{equation}
  \label{eq:JOmegaInverse}
  \left( \J_\Omega^{-1} w \right) \left(x+ \left( \begin{smallmatrix} \Lambda j \\ 0 \end{smallmatrix} \right) \right)
  = \left[ \frac{\Lambda}{2\pi} \right]^{1/2} 
  \int_\Wast w(\alpha, x) \, e^{\i \, \Lambda j \alpha} \d{\alpha}
  \qquad \text{for }  x \in \Omega_H^\Lambda, \, j\in\Z .
\end{equation}
% As in Theorem~\ref{th:BlochR}, $w(\alpha, \cdot)$ denotes the $\alpha$-quasiperiodic extension to $w \in H^r_\p(\Wast; \widetilde{H}^s_\alpha(\Omega_H^\Lambda))$. 
\end{theorem}

Finally, we introduce Sobolev spaces on $\Gamma_H$ and $\Gamma_H^\Lambda = \{  x \in \Gamma_H: \, x_1 \in \W \} \subset \Gamma_H$ by identifying $\Gamma_H$ with $\R$ and $\Gamma_H^\Lambda$ with $\W$. 
The resulting spaces are then denoted by $H^s_r(\Gamma_H)$, $H^s_\alpha(\Gamma_H^\Lambda)$, and $H^r_\p(\Wast; H^s_\alpha(\Gamma_H^\Lambda))$ for $s,r\in \R$.

\section{Quasiperiodic Scattering Problems}\label{se:qpScatt}
The Bloch transform of a solution $u$ to the variational formulation~\eqref{eq:varFormHEScal} of the surface scattering problem involving the periodic surface $\Gamma$ solves the variational formulation of a corresponding quasiperiodic  scattering problem. 
This variational formulation relies on a periodic Dirichlet-to-Neumann operator $T_\alpha^+$ on $\Gamma_H^\Lambda\subset \Gamma_H$ that is continuous from $H^{s}_\alpha(\Gamma_H^\Lambda)$ into $H^{s+1}_\alpha(\Gamma_H^\Lambda)$ for all $s\in \R$, 
\begin{equation}\label{eq:perioDtN}
   T_\alpha^+ \left( \varphi \right)\Big|_{\Gamma_H^\Lambda} 
   = \bigg[\i \sum_{j\in\Z} \sqrt{k^2 - |\Lambda^* j + \alpha|^2} \, \hat{\varphi}(j) \, e^{\i (\Lambda^* j + \alpha)  x_1} \bigg]\bigg|_{\Gamma_H^\Lambda}
   \quad \text{for } 
   \varphi = \sum_{j\in\Z} \hat{\varphi}(j) e^{\i (\Lambda^* j + \alpha)  x_1}.
\end{equation}
Obviously, $T_\alpha^+$ corresponds to $T^+$ from~\eqref{eq:T}. 
We further introduce a bounded sesquilinear form $a_\alpha $ on $\widetilde{H}^1_\alpha(\Omega_H^\Lambda) \times \widetilde{H}^1_\alpha(\Omega_H^\Lambda)$, 
\[
  a_\alpha(w,v) 
  := \int_{\Omega_H^\Lambda} \Big[ \nabla_x w \cdot \nabla_x \overline{v} - k^2 w(\alpha, \cdot) \, \overline{v} \Big] \d{x} 
  - \int_{\Gamma_H^\Lambda} T_\alpha^+ \left( w|_{\Gamma_H^\Lambda} \right) \, \overline{v}|_{\Gamma_H^\Lambda} \dS, 
\]
and state an equivalence result that can be shown along the lines of Theorem~9 in~\cite{Lechl2016}. 

\begin{theorem}\label{th:equiScalarPerio}
Suppose $u^i$ is a twice differentiable solution of the Helmholtz equation in $\Omega$ that belongs to $H^1_r(\Omega_H)$ for some $r \in \R$.  
Then $u \in \widetilde{H}^1_r(\Omega_H^\Lambda)$ solves~\eqref{eq:varFormHEScal} if and only if $w := \J_\Omega u \in H^r_\p(\Wast; \widetilde{H}^1_\alpha(\Omega_H^\Lambda))$ solves   
\begin{equation}\label{eq:heAlpha}
  a_\alpha(w(\alpha,\cdot),v)
  = \int_{\Gamma_H^\Lambda} \left[ \frac{\partial }{\partial \nu} \J_\Omega u^i (\alpha,\cdot)- T_\alpha^+ \left[\J_{\Gamma_H} u^i (\alpha,\cdot)\big|_{\Gamma_H^\Lambda} \right] \right] \overline{v} \dS  
\end{equation}
for almost every $\alpha \in \Wast$ and all $v \in \widetilde{H}^1_\alpha(\Omega_H^\Lambda)$. 
%
% \high{Maybe suppress the following (not strictly necessary ?):} 
%
If we denote the Fourier coefficients of $w(\alpha,\cdot)|_{\Gamma_H^\Lambda}$ as in~\eqref{eq:perioDtN} by $\hat{w}(\alpha,j)$, then there holds that $\hat{w}(\alpha,j) = \dhat{u}(\Lambda^* j + \alpha, H)$, see~\eqref{eq:fourierTrafo}. 
Extending $w(\alpha, \cdot)$ by $\alpha$-quasiperiodicity to $\Omega_H$, these coefficients allow to further extend  
\begin{equation}\label{eq:URCAlpha}
  w(\alpha, x) = 
  \Lambda^{-1/2} \, \sum_{j\in\Z} \hat{w}(\alpha,j) \, e^{\i (\Lambda^* j + \alpha ) x_1 + \i \beta(j) x_2} \, 
  \ \text{ for $x_2 \geq H$, } \beta(j) = \sqrt{k^2 - | \Lambda^* j + \alpha|^2},
\end{equation}
which is an $\alpha$-quasiperiodic weak solution of $\Delta_x w(\alpha,\cdot) + k^2 w(\alpha, \cdot) =0$ in $\Omega$ that belongs to $\widetilde{H}^1_r(\Omega_{H'}^\Lambda)$ for all $H'\geq H$.
\end{theorem}

\begin{remark}
  The normal derivative and the Bloch transform commute, such that we could replace $(\partial/\partial\nu) \J_\Omega u^i $ in~\eqref{eq:heAlpha} by $\J_{\Gamma_H} (\partial u^i/\partial\nu)$. 
\end{remark}

The compact embedding of $\widetilde{H}^1_\alpha(\Omega_H^\Lambda)$ in $L^2(\Omega_H^\Lambda)$ implies that $a_\alpha $ satisfies a G\r{a}rding inequality: 
If we abbreviate $\hat{v}(j) = \widehat{(v|_{\Gamma_H^\Lambda})}(j)$ for $j\in \Z$, then 
\begin{equation}\label{eq:aux414}
  \Re a_\alpha(v,v)  
  = \int_{\Omega_H^\Lambda} \left[ \left| \nabla v \right|^2 - k^2 |v|^2 \right] \d{x} 
  - \sum_{j\in\Z} \Re (\i\beta(j)(k,\alpha)) \, |\hat{v}(j)|^2,
\end{equation}
where $\Re (\i\beta(j)(k,\alpha)) \leq 0$ for all $j \in\Z$ is strictly positive if and only if $| \Lambda^* j + \alpha| > k$ and vanishes else.
Thus, existence of solution to~\eqref{eq:heAlpha} for fixed $\alpha \in \W$ follows from uniqueness. 
Uniqueness of solution in turn is well-known under the assumption that $\Gamma$ is graph of a Lipschitz continuous function, see~\cite[Cor.~3.4]{Elsch2002} or~\cite[Sec.~3.5]{Bonne1994}. 
The latter assumption is equivalent to assume that the diffeomorphism $\zeta$ defining $\Gamma$ satisfies $\zeta(x_1,0) = f(x_1)$ for some $\Lambda$-periodic Lipschitz continuous function on the real line. 
Under this assumption, the solution operators $A_\alpha: \, \big(\widetilde{H}^1_\alpha(\Omega_H^\Lambda)\big)^* \to \widetilde{H}^1_\alpha(\Omega_H^\Lambda)$, defined on the dual of $\widetilde{H}^1_\alpha(\Omega_H^\Lambda)$ by $A_\alpha F = w$ if $w \in \widetilde{H}^1_\alpha(\Omega_H^\Lambda)$ solves $a_\alpha(w,v) = F(v)$ for all $v \in \widetilde{H}^1_\alpha(\Omega_H^\Lambda)$, are invertible and moreover uniformly bounded in $\alpha \in \Wast$.   

\begin{lemma}\label{th:uniqueAlpha}
If $\Gamma$ is graph of a Lipschitz continuous function, then~\eqref{eq:heAlpha} is solvable for all $(k_\ast,\alpha) \in (0,\infty) \times \Wast$ and the solution operators $A_\alpha$ are uniformly bounded in $\alpha$.  
\end{lemma}
\begin{proof}
The sesquilinear form $a_\alpha $ merely depends on $\alpha$ via the exterior Dirichlet-to-Neumann operator,
\[
  (w,v) \mapsto \int_{\Gamma_H^\Lambda} T_\alpha^+ w \, \overline{v} \dS 
  = \i \sum_{j\in\Z} \sqrt{k^2 - |\Lambda^* j + \alpha|^2} \, \hat{w}(\alpha,j) \overline{\hat{v}(j)},
\]
which is a $\Lambda^*$-periodic function in $\alpha$ and, moreover, continuous in $\alpha$, as we show now:  
For $j\in\Z$ such that $|\Lambda^* j + \alpha| > 2k$ for all $\alpha \in \Wast$, the square root functions $\alpha \mapsto \sqrt{k^2 - | \Lambda^* j + \alpha |^2}$ are infinitely smooth.  
For the remaining finitely many $j$, the finitely many functions $t \mapsto \sqrt{k^2 - |\Lambda^* j + \alpha|^2}$ are all H\"older continuous (with same H\"older exponent).
Convergence of the series defining $T_\alpha^+$ in the operator norm of $\mathcal{L}(H^{1/2}_\alpha(\Gamma_H^\Lambda),H^{-1/2}_\alpha(\Gamma_H^\Lambda))$ (which is independent of $\alpha$) ensures that $(w(\alpha,\cdot),v) \mapsto (T_\alpha^+ w(\alpha, \cdot),v)$ and $a_\alpha $ both depend continuously on $\alpha \in \Wast$. 
Consequently, the above-defined solution operator $A_\alpha$ depends continuously on $\alpha \in \Wast$, such that its inverse $A_\alpha^{-1}$ depends continuously on $\alpha$ as well. 
In particular, the operator norms $\| A_\alpha^{-1} \|$ are uniformly bounded in $\alpha \in \Wast$.  
\end{proof}

We next show a regularity result for the solution $w$ to the periodic problem~\eqref{eq:heAlpha} in $\alpha$ that in particular implies that $\alpha \mapsto w(\alpha, \cdot)$ is continuous. 
To this end, we introduce the spaces 
\[
  W^{1,p}_\p(\Wast; \widetilde{H}^1_\alpha(\Omega_H^\Lambda)), \qquad 1 \leq p < \infty, 
\]
as restrictions to $\Wast \times \Omega_H^\Lambda$ of those distributions in $\mathcal{D}_{\Lambda}'(\R^2)$ for which the following norm is finite, 
\[
  \| w \|_{W^{1,p}_\p(\Wast; \widetilde{H}^1_\alpha(\Omega_H^\Lambda))}^p 
  = \int_{\Wast} \left[ \| w(\alpha,\cdot) \|_{\widetilde{H}^1_\alpha(\Omega_H^\Lambda))}^p  + \| \partial_\alpha w(\alpha,\cdot) \|_{\widetilde{H}^1_\alpha(\Omega_H^\Lambda))}^p \right] \d{\alpha}.
\]
% If $w \in W^{1,p}_\p(\Wast; \widetilde{H}^1_\alpha(\Omega_H^\Lambda))$, then well-known trace results for Sobolev spaces imply that $w|_{\Gamma_H^\Lambda}$ belongs to $W^{1,p}_\p(\Wast; H^{1/2}_\alpha(\Gamma_H^\Lambda))$.

\begin{theorem} \label{th:exiSolScalPerio}
Assume that $\Gamma$ is graph of a Lipschitz continuous function.\\[1mm]
(a) If $u^i \in H^1(\Omega_H)$, then the solution $w=w(\alpha,\cdot)$ to~\eqref{eq:heAlpha} belongs to $L^2(\Wast; \widetilde{H}^1_\alpha(\Omega_H^\Lambda))$ and the solution $u = \J_\Omega^{-1} w$ to~\eqref{eq:varFormHEScal} belongs to $H^{1}(\Omega_H)$.\\[1mm] 
(b) If $u^i \in H^1(\Omega_H)$ satisfies for some $p \in [1,2)$ that 
\begin{equation}\label{eq:assSource}
  \J_\Omega u^i \in W^{1,p}_\p(\Wast; H^1_\alpha(\Omega_H^\Lambda))
  \quad \text{and} \quad 
  \sup_{\alpha \in \Wast} \| \J_\Omega u^i(\alpha, \cdot) \|_{H^1_\alpha(\Omega_H^\Lambda)} < \infty, 
\end{equation}
then the solution $w=w(\alpha,\cdot)$ to~\eqref{eq:heAlpha} belongs to $W^{1,p}_\p(\Wast; \widetilde{H}^1_\alpha(\Omega_H^\Lambda))$.
\end{theorem}

\begin{remark}\label{rem:2}
If $u^i \in H^1_r(\Omega_H)$ with $1/2<r<1$, then $\mathcal{J}_\Omega u^i$ belongs to $H^r_\p(\Wast; H^1_\alpha(\Omega_H^\Lambda))$ by Theorem~\ref{th:BlochR}. 
Sobolev's embedding result for functions defined in $\Wast \times \Gamma_H^\Lambda$ hence shows that $\alpha \mapsto \mathcal{J}_{\Gamma_H} u^i|_{\Gamma_H}(\alpha,\cdot)$ is continuous from $\Wast$ into $H^{1/2}_\alpha(\Gamma_H^\Lambda)$. 
In particular, $\max_{\alpha \in \Wast} \| \mathcal{J}_\Omega u^i(\alpha,\cdot) \|_{H^1_\alpha(\Omega_H^\Lambda)} < \infty$ then is finite and~\eqref{eq:assSource} is satisfied. 
This applies for instance if $u^i$ is the Dirichlet Green's function $G(\cdot,y)$ for the half-space, defined by  
\begin{equation}\label{eq:greenFunc}
  G(x,y) = \frac{\i}{4} \left[ H^{(1)}_0(k|x-y|) - H^{(1)}_0(k|x-y'|) \right], \quad x \not = y \in \R^2,
\end{equation}
for a source point $y \in \Omega$ with reflection $y'=(y_1,-y_2)^\top$ at $\Gamma_0$. 
Due to this mirrored point, $G(\cdot,y)$ decays faster than the single point source as $|x_1|\to\infty$. 
Indeed, $G(\cdot,y)$, as well as all its first-order partial derivatives, decays as $(1+|x_1|^2)^{-3/4}$ when $|x_1| \to \infty$, see~\cite{Chand1996}. 
Hence, for $y =(y_1,y_2)^\top$ with $y_2>0$ and $r<1$, $G(\cdot,y)$ satisfies that 
\begin{equation}\label{eq:GNorm}
  \int_{\R \times [0,H]} (1+|x_1|^2)^{r} \left[ |G(x,y)|^2 + |\nabla_x G(x,y)|^2 \right] \d{x} 
  \leq C \int_{\R \times [0,H]} (1+|x_1|^2)^{r-3/2} \d{x}  < \infty
\end{equation}
Note further that also incident Gaussian beams, modeled by incident downwards-traveling Herglotz wave functions, 
\[
  v_g(x) = \int_{-\pi/2}^{\pi/2} e^{\i k \, [\sin(\vartheta) \, x_1 - \cos(\vartheta) \, x_2]} g(\vartheta) \d{\vartheta},
  \quad x \in \R^2, 
\]
belong to the weighted spaces $H^1_r(\Omega_H)$ under suitable conditions on $g$, see~\cite{Lechl2016, Lechl2015e}.
\end{remark}

\begin{proof}
(a) Recall that the normal derivative on $\Gamma_H^\Lambda$ of the solution $\J_\Omega u^i(\alpha,\cdot)$ of the quasiperiodic Helmholtz equation in $\Omega_H^\Lambda$ satisfies the trace bound $\| (\partial [\J_\Omega u^i (\alpha,\cdot)]/\partial \nu) \|_{H^{-1/2}_\alpha(\Gamma_H^\Lambda)} \leq C \| \J_\Omega u^i(\alpha,\cdot) \|_{H^1_\alpha(\Omega_H^\Lambda)}$ for some $C>0$ independent of $\alpha \in \Wast$. 
The uniform bound of the inverses $\| A_\alpha^{-1} \|$ shown in Lemma~\ref{th:uniqueAlpha} together with the trace theorem in $H^1_\alpha(\Omega_H^\Lambda)$ hence implies for the solutions $w(\alpha,\cdot)$ to~\eqref{eq:heAlpha} that 
\begin{align}
  \| w & \|_{L^2(\W^*; \widetilde{H}^1_\alpha(\Omega_H^\Lambda))}^2  
  = \int_\Wast \| w(\alpha,\cdot) \|_{\widetilde{H}^1_\alpha(\Omega_H^\Lambda)}^2 \d{\alpha} \nonumber \\
  & \quad \leq C \int_\Wast \left[ \left\| \frac{\partial }{\partial \nu}\J_{\Gamma_H} u^i (\alpha,\cdot) \right\|_{H^{-1/2}_\alpha(\Gamma_H^\Lambda)}^2 + \| T^+_\alpha \J_{\Gamma_H} u^i(\alpha,\cdot) \|_{H^{-1/2}_\alpha(\Gamma_H^\Lambda)}^2 \right]\d{\alpha} \label{eq:aux500} \\
  & \quad\leq C \int_\Wast \left[ \| \J_\Omega u^i(\alpha,\cdot) \|_{H^1_\alpha(\Omega_H^\Lambda)}^2 + \| \J_{\Gamma_H} u^i(\alpha,\cdot) \|_{H^{1/2}_\alpha(\Gamma_H^\Lambda)}^2 \right]\d{\alpha} 
  \leq C \| u^i \|_{H^1(\Omega_H)}^2, \nonumber
\end{align}
where we exploited that the operator norm of $T_\alpha$ can be uniformly bounded in $\alpha\in \Wast$ by $\sup_{j\in\Z}\big[ \big|k^2-|\Lambda^* j-\alpha|^2 \, \big|/|1+|j|^2| \big]^{1/2} \leq C(k)$. 

(b) 
%For $f \in H^{-1/2}_1(\Gamma)$ there holds that $\J_{\Gamma} f \in H^1_\p(\Wast; H^{-1/2}_\alpha(\Gamma_\Lambda))$. 
The assumed bound on the right-hand side of~\eqref{eq:assSource} implies by the uniform bound for the inverses $A_\alpha^{-1}$ shown in Lemma~\ref{th:uniqueAlpha} that for arbitrary $\alpha\in\Wast$ there holds 
\[
  \| w(\alpha,\cdot) \|_{\widetilde{H}^1_\alpha(\Omega_H^\Lambda)} 
  % \leq C \left[\left\| \frac{\partial }{\partial \nu}\J_{\Gamma_H} u^i (\alpha,\cdot) \right\|_{H^{-1/2}_\alpha(\Gamma_H^\Lambda)} + \| \J_{\Gamma_H} u^i (\alpha,\cdot) \|_{H^{1/2}_\alpha(\Gamma_H^\Lambda)} \right]
  \leq C \| \J_\Omega u^i(\alpha,\cdot) \|_{H^1_\alpha(\Omega_H^\Lambda)}^2
\]
for some constant $C>0$ independent of $\alpha \in \Wast$, see~\eqref{eq:aux500}.  
Additionally, the radiating quasiperiodic extension of $w(\alpha,\cdot)$ to all of $\Omega$ by~\eqref{eq:URCAlpha} satisfies the Helmholtz equation $\Delta w + k^2 w = 0$ in $\Omega$ and is hence a smooth function in $\Omega$ by standard elliptic regularity results, see, e.g.,~\cite{McLea2000}. 
In particular, $w(\alpha,\cdot)|_{\Gamma_H^\Lambda}$ belongs to all Sobolev spaces $H^s_\alpha(\Gamma_H^\Lambda)$ and satisfies the bounds
\begin{equation}\label{eq:regW}
  \sum_{j\in\Z} (1+|j|^2)^s |\hat{w}(\alpha,j)|^2 
  \| w(\alpha,\cdot)|_{\Gamma_H^\Lambda} \|_{H^s_\alpha(\Gamma_H^\Lambda)}^2
  \leq C(s)^2 \| \J_\Omega u^i(\alpha,\cdot) \|_{H^1_\alpha(\Omega_H^\Lambda)}^2
  \text{ for all $s \geq 0$.} 
\end{equation}
(Again, $\hat{w}(\alpha,j)$ denotes the $j$th Fourier coefficient of the restriction of $w(\alpha,\cdot)$ to $\Gamma_H^\Lambda$.)
The latter bounds imply existence of a constant $C(s)$ such that 
\begin{equation}\label{eq:regW2}
  |\hat{w}(\alpha,j)| \leq C(s) (1+|j|^2)^{-s} \| \J_\Omega u^i(\alpha,\cdot) \|_{H^1_\alpha(\Omega_H^\Lambda)} 
  \quad \text{for all } j\in\Z.
\end{equation}

Let us now abbreviate the variational formulation~\eqref{eq:heAlpha} for $w(\alpha,\cdot)$ by $a_\alpha(w(\alpha,\cdot),v) = F_\alpha(v)$ for all $v \in \widetilde{H}^1_\alpha(\Omega_H^\Lambda)$. 
The solutions $w(\alpha,\cdot)$ to~\eqref{eq:heAlpha} then possess a (formal) derivative $w' = \partial_{\alpha} w(\alpha,\cdot)$ with respect to $\alpha$ that satisfies $a_\alpha(w',v) = F_{\alpha}^{(1)} (v) + F_{\alpha}^{(2)}(v)$ for all $v\in \widetilde{H}^1_\alpha(\Omega_H^\Lambda)$, with right-hand sides 
\begin{align}
   F_{\alpha}^{(2)}(v)
  & = \i \sum_{j\in\Z} \frac{\Lambda^* j + \alpha}{\sqrt{k^2 - |\Lambda^* j + \alpha |^2}} \, \hat{w}(\alpha,j) \, \overline{\hat{v}(j)}, \label{eq:RHSDeriAlpha}
\end{align}
where $\hat{v}(j)$ denotes the $j$th Fourier coefficient of the restriction of $v$ to $\Gamma_H^\Lambda$, see~\eqref{eq:perioDtN}, and 
\begin{align}
  F_{\alpha}^{(1)} (v) 
  & = \partial_{\alpha} F_\alpha(v) 
  % = \int_{\Gamma_H^\Lambda} \partial_{\alpha} \left[ \frac{\partial }{\partial \nu} \J_\Omega u^i (\alpha,\cdot)- T_\alpha^+ \left[\J_{\Gamma_H} u^i (\alpha,\cdot)\big|_{\Gamma_H^\Lambda} \right] \right] \, \overline{v} \dS  \\ & 
  = \int_{\Gamma_H^\Lambda} \left[ \frac{\partial }{\partial \nu} \partial_{\alpha} \J_\Omega u^i (\alpha,\cdot)- T_\alpha^+ \left[ \big[\partial_{\alpha}  \J_{\Omega} u^i (\alpha,\cdot) \big] \big|_{\Gamma_H^\Lambda} \right] \right]\, \overline{v} \dS \label{eq:aux541}\\
  & \hspace*{5cm} + \i \sum_{j\in\Z} \frac{(\Lambda^* j + \alpha)_{1,2}}{\sqrt{k^2 - |\Lambda^* j  + \alpha |^2}} \, \widehat{(\J_{\Omega} u^i)}(\alpha,j) \, \overline{\hat{v}(j)}, \label{eq:aux542}
\end{align}
where $\widehat{(\J_{\Omega} u^i)}(\alpha,j)$ denotes the $j$th Fourier coefficient of $\J_{\Omega} u^i(\alpha,\cdot)|_{\Gamma_\Lambda^H}$ as defined in~\eqref{eq:perioDtN}.

We show next that the right-hand side $F_{\alpha}^{(2)}$ is well-defined and bounded for all $\alpha \in \Wast \setminus E_{\Lambda^*}$ where $E_{\Lambda^*} = \{ \alpha \in \Wast: \, |\Lambda^* j  + \alpha | = k$ for some $j \in \Z \}$.
Indeed, as the equation $|\Lambda^* j + \alpha | = k$ can only be satisfied for finitely many $j \in\Z$; a necessary condition is that $j \in I_1 = \{ j\in\Z: \, |\Lambda^* j| \leq k + |\Lambda^*| \}$.
Thus, $E_{\Lambda^*}$ consists of finitely many points in the interval $\Wast$. 
As, moreover, 
\[
  \frac{\big| (\Lambda^* j + \alpha)_{1,2} \big|}{| k^2 - |\Lambda^* j + \alpha|^2 |^{1/2}} 
  \leq 
  \frac{|\Lambda^* j + \alpha|}{| k^2 - |\Lambda^* j + \alpha|^2 |^{1/2}} 
  \to 1  
  \quad \text{as } |j| \to \infty, 
\]
we conclude by~\eqref{eq:regW2} with $s=2$ and the continuity estimate $|\hat{v}(j)| \leq C\|v \|_{\widetilde{H}^1_\alpha(\Omega_H^\Lambda)}$ that for all $v\in \widetilde{H}^1_\alpha(\Omega_H^\Lambda)$ with $\| v \|_{\widetilde{H}^1_\alpha(\Omega_H^\Lambda)} = 1$ and all $\alpha \in \Wast \setminus E_{\Lambda^*}$ there holds that  
\begin{align}
  \big| F_{\alpha}^{(2)} (v) \big|
  &\leq 
  C \sum_{j\in\Z} \frac{|\Lambda^* j + \alpha| \, (1+|j|^2)^{-2} }{| k^2 - |\Lambda^* j + \alpha|^2 |^{1/2}} \| \J_{\Gamma_H} u^i(\alpha,\cdot) \|_{H^{1/2}_\alpha(\Gamma_H^\Lambda)} \nonumber \\
  & 
  \leq 
  C  \bigg(\sum_{j\in I_1} + \sum_{j\in \Z \setminus I_1}\bigg) \frac{|\Lambda^* j + \alpha| \, (1+|j|^2)^{-2} }{| k^2 - |\Lambda^* j + \alpha|^2 |^{1/2}} \| \J_\Omega u^i(\alpha,\cdot) \|_{H^1_\alpha(\Omega_H^\Lambda)} \label{eq:aux564}\\
  & 
  \leq 
  \bigg[ C_1(\Lambda,k) + C_2(\Lambda,k) \sum_{j: \, |\Lambda^* j| \leq k + |\Lambda^*|}  \frac{1}{| k^2 - |\Lambda^* j + \alpha|^2 |^{1/2}} \bigg] \| \J_\Omega u^i(\alpha,\cdot) \|_{H^1_\alpha(\Omega_H^\Lambda)} \nonumber
\end{align}
for constants $C_{1,2}$ depending on $\Lambda$ and $k$ but independent of $\alpha$ and $u^i$.

Concerning $F_{\alpha}^{(1)}$, let us first note that $u^i$ is by assumption a twice continuously differentiable solution to the Helmholtz equation in $\Omega$, and hence a real-analytic function, such that $\J_\Omega u^i(\alpha,\cdot)$ belongs to all Sobolev spaces $H^s_\alpha(\Omega_H^\Lambda)$ for arbitrary $s \geq 0$.  
Consequently, the bounds~\eqref{eq:regW2} hold analogously for $\J_\Omega u^i(\alpha,\cdot)$ instead of $w(\alpha,\cdot)$, i.e.,
\begin{equation}\label{eq:regW2NochMal}
  |\widehat{(\J_{\Omega} u^i)}(\alpha,j)| \leq C(s) (1+|j|^2)^{-s} \| \J_\Omega u^i(\alpha,\cdot) \|_{H^1_\alpha(\Omega_H^\Lambda)} 
  \quad \text{for all } j\in\Z.
\end{equation}
In particular, the third term of $F_{\alpha}^{(1)}$ from~\eqref{eq:aux542} is well-defined and bounded in $v$ for all $\alpha \in \Wast \setminus E_{\Lambda^*}$ by the same arguments we used for $F_{\alpha}^{(2)}$. 
In particular, this third term from~\eqref{eq:aux542} can also be bounded as in~\eqref{eq:aux564}. 
Further, the first two terms of $F_{\alpha}^{(1)}$ from~\eqref{eq:aux541} are clearly well-defined and bounded in $v$ as $\partial_{\alpha} \J_{\Omega} u^i(\alpha,\cdot)$ belongs to $H^1_\alpha(\Omega_\Lambda^H)$ due to~\eqref{eq:assSource}.  

As the right-hand sides $F_{\alpha}^{(1,2)}$ are hence well-defined in $\big(\widetilde{H}^1_\alpha(\Omega_H^\Lambda)\big)^*$ for all $\alpha \in \Wast$ except finitely many point in $E_{\Lambda^*}$, and as the solutions operators $A_\alpha^{-1}$ are uniformly bounded in $\alpha \in \Wast$, we conclude that $\| w'(\alpha,\cdot) \|_{\widetilde{H}^1_\alpha(\Omega_H^\Lambda)} \leq C \big[ \| F_{\alpha}^{(1)} \| + \| F_{\alpha}^{(2)} \| \big]$ for all $\alpha \in \Wast \setminus E_{\Lambda^*}$.
Now, recall from part (a) of this proof the trace bound 
\[
  \bigg\| \frac{\partial }{\partial \nu} \big[\partial_{\alpha} \J_\Omega u^i (\alpha,\cdot) \big] \bigg\|_{H^{-1/2}_\alpha(\Gamma_H^\Lambda)}
  = \bigg\| \partial_{\alpha}  \frac{\partial }{\partial \nu} \left[ \J_\Omega u^i (\alpha,\cdot) \right]\bigg\|%_{H^{-1/2}_\alpha(\Gamma_H^\Lambda)}
  \leq C \| \partial_{\alpha} \J_\Omega u^i (\alpha,\cdot) \|_{H^1_\alpha(\Omega_H^\Lambda)}, 
\]
the uniform boundedness of the operator norms of $T_\alpha^+: \, H^{1/2}_\alpha(\Gamma_H^\Lambda) \to H^{-1/2}_\alpha(\Gamma_H^\Lambda)$, and the trace theorem stating that $\| \partial_{\alpha} \J_{\Gamma_H} u^i (\alpha,\cdot) \|_{H^{-1/2}_\alpha(\Gamma_H^\Lambda)} \leq C \| \partial_{\alpha} \J_\Omega u^i (\alpha,\cdot) \|_{H^1_\alpha(\Omega_H^\Lambda)}$, again with a uniform constant in $\alpha \in \Wast$. 
Thus, the estimate  $(a_1+ \dots+a_n)^p \leq n^{p-1} (a_1^p+ \dots+a_n^p) $ for non-negative numbers $a_1, \dots, a_n$ and $p\geq1$ due to Jensen's inequality implies that 
\begin{align*}
  \big\| F_{\alpha}^{(1)} + F_{\alpha}^{(2)} \big\|_{\big(\widetilde{H}^1_\alpha(\Omega_H^\Lambda)\big)^*}^p
  \leq C \left[ \| \partial_{\alpha} \J_\Omega u^i (\alpha,\cdot) \|_{H^1_\alpha(\Omega_H^\Lambda)}^p +  \sum_{j \in I_1}  \frac{\| \J_\Omega u^i(\alpha,\cdot) \|_{H^1_\alpha(\Omega_H^\Lambda)}^p}{| k^2 - |\Lambda^* j + \alpha|^2 |^{p/2}}  \right],
\end{align*}
for $C>0$ independent of $\alpha\in\Wast \setminus E_{\Lambda^*}$. 
Consequently, 
\begin{align}
  \| w' &  \|_{L^p(\Wast; \widetilde{H}^1_\alpha(\Omega_H^\Lambda))}^p  
  = 
  \int_\Wast \| w' (\alpha, \cdot ) \|_{\widetilde{H}^1_\alpha(\Omega_H^\Lambda)}^p \d{\alpha} \nonumber \\
  & \leq C \int_\Wast \left[ \big\| F_{\alpha}^{(1)} + F_{\alpha}^{(2)} \big\|_{\big(\widetilde{H}^1_\alpha(\Omega_H^\Lambda)\big)^*}^p \right] \d{\alpha} 
  % & \leq C \int_\Wast \bigg[ \| \partial_{\alpha} \J_\Omega u^i (\alpha,\cdot) \|_{H^1_\alpha(\Omega_H^\Lambda)}^p +  \sum_{j: \, |\Lambda^* j| \leq k + |\Lambda^*|}  \frac{1}{| k^2 - |\Lambda^* j - \alpha|^2 |^{p/2}} \| \J_\Omega u^i(\alpha,\cdot) \|_{H^1_\alpha(\Omega_H^\Lambda)}^p\bigg]   \d{\alpha} \nonumber \\
   \leq C_1 \| \J_\Omega u^i \|_{W^{1,p}_\p(\Wast; H^1_\alpha(\Omega_H^\Lambda))}^p \nonumber \\
   & \qquad\qquad + C_2 \sum_{j\in I_1} \int_\Wast | k^2 - |\Lambda^* j + \alpha|^2 |^{-p/2} \d{\alpha} \sup_{\alpha \in \Wast} \| \J_\Omega u^i(\alpha,\cdot) \|_{H^1_\alpha(\Omega_H^\Lambda)}^p . \label{eq:aux1146}
\end{align}
Note that the last supremum is finite by~\eqref{eq:assSource}. 
The integrand in the last line of~\eqref{eq:aux1146} is singular at the finitely many points in $E_{\Lambda^*}$, and the representation 
\[
  \alpha \mapsto | k^2 - |\Lambda^* j + \alpha|^2 |^{-p/2}
  = [k+|\Lambda^* j + \alpha|]^{-p/2} |k-|\Lambda^* j + \alpha||^{-p/2}
\]
shows that the singularity is of the order $p/2$, which is strictly less than one for $p \in [1,2)$. 
Under the assumption that $p \in [1,2)$, the latter function is hence integrable in $\Wast$. 
Thus, the integral $\int_\Wast | k^2 - |\Lambda^* j + \alpha|^2 |^{-p/2} \d{\alpha}$ exists as an improper integral from~\eqref{eq:aux1146} takes a finite value, such that $w' \in L^p(\Wast; \widetilde{H}^1_\alpha(\Omega_H^\Lambda))$ and $w \in W^{1,p}_\p(\Wast; \widetilde{H}^1_\alpha(\Omega_H^\Lambda))$.
\end{proof}

\begin{remark}
The proof of Theorem~\ref{th:exiSolScalPerio}(b) indicates that the range of $p\in[1,2)$ is sharp. 
A general solution theory for~\eqref{eq:heAlpha} in $W^{1,2}_\p(\Wast; \widetilde{H}^1_\alpha(\Omega_H^\Lambda)) = H^1_\p(\Wast; \widetilde{H}^1_\alpha(\Omega_H^\Lambda))$ actually cannot be expected, as we know from~\cite{Chand2010} that for general aperiodic surfaces the surface scattering problem~\eqref{eq:varFormHEScal} does merely possess unique solutions in $\widetilde{H}^1_r(\Omega_H)$ for $|r|<1$.
% (b) A sufficient ingredient for the proof of Theorem~\ref{th:exiSolScalPerio}(a) is the uniform boundedness of the solution operators $A_\alpha$ for all quasiperiodicities $\alpha \in \Wast$. 
\end{remark}

\section{Error Estimates for a Discrete Inverse Bloch Transform}\label{se:errInvBloch}
We now aim to use the forward and inverse Bloch transform to design a numerical scheme approximating the solution of variational formulation~\eqref{eq:varFormHEScal} of the surface scattering problem~\eqref{eq:HE}-\eqref{eq:URC} that is posed in the unbounded layer $\Omega_H$. 
Naturally, the idea is to solve several quasiperiodic scattering problems~\eqref{eq:heAlpha} posed in the bounded domain $\Omega_H^\Lambda$ approximately in a finite-dimensional space, in order to get an approximation to the solution of the variational formulation~\eqref{eq:varFormHEScal} of~\eqref{eq:HE}-\eqref{eq:URC} via a discretized inverse Bloch transform. 
Thus, we discretize the interval $\Wast = (-\pi/\Lambda^*, \pi/\Lambda^*]$ uniformly by $N \in\N$ discretization points 
\[
  \alpha_j = -\pi/\Lambda^*+2\pi j/(N\Lambda^*) = -\pi/\Lambda^*+2 j \Lambda/(2N) 
  \qquad \text{ for } j=1,\dots,N.
\]
The first step of the scheme is to compute the Bloch transform $J_{\Omega}u^i(\alpha_j,\cdot)$ of the incident field at the discretization points $\alpha_j$ and to solve the quasiperiodic variational problems~\eqref{eq:heAlpha} for all $\alpha_j$, $j=1,\dots,N$, e.g., by the finite element method (any discretization method that yields convergent discrete solutions could be used here). 
If we denote by $w(\alpha,\cdot)\in H^1_\alpha(\Omega_H^\Lambda)$ the exact solution of the variational problem \eqref{eq:heAlpha} and by $w_h(\alpha,\cdot)$ the discrete solution, then our error analysis basically requires that $\| w_h(\alpha,\cdot) - w(\alpha,\cdot)\|$ tends to zero in some approximate norm for all $\alpha \in \Wast$. 
In particular, the Dirichlet scattering problem considered here can be replaced by any problem such that uniqueness and existence holds in a suitable variational space. 

The second step of the scheme is to obtain an approximate solution to the scattering problem \eqref{eq:HE}-\eqref{eq:URC} via the inverse Bloch transform. 
As the inverse Bloch transform, see \eqref{eq:JOmegaInverse}, is for points $x \in \Omega_H^\Lambda$ merely an integration on $\Wast$, we use a simple integration rule (the trapezoidal rule) for numerical approximation, 
\begin{equation}
\label{numeric:inverseBloch}
    \Big( \J_{\Omega,N}^{-1} \{ w(\alpha_j, \cdot) \}_{j=1}^N \Big) (x)
    = \left[ \frac{\Lambda}{2\pi} \right]^{1/2} \frac{2\pi}{N\Lambda^*}    \sum_{j=1}^N w(\alpha_j, x), 
    \qquad x \in \Omega_H^\Lambda.
\end{equation}
% As for~\eqref{eq:JOmegaInverse}, we put emphasis on the fact that in the latter equation the function $x \mapsto w(\alpha_j, x)$ is the quasiperiodic extension of $w(\alpha_j, \cdot) \in \widetilde{H}^1_{\alpha_j}(\Omega_H^\Lambda)$ to all of $\Omega_H$. 
The numerical approximation of the solution $u \in \widetilde{H}^1_r(\Omega_H)$ to~\eqref{eq:varFormHEScal} then equals  \begin{equation}
\label{numeric:solution}
  u_{N,h}=\left( \J_{\Omega,N}^{-1} \{ w_h(\alpha_j, \cdot) \}_{j=1}^N \right) (x)
= \left[ \frac{\Lambda}{2\pi} \right]^{1/2} \frac{2\pi}{N\Lambda^*}    \sum_{j=1}^N w_h(\alpha_j, x),
    \qquad x \in \Omega_H^\Lambda.
\end{equation}

\begin{remark}
(a) For simplicity, we abbreviate $\J_{\Omega,N}^{-1} \{ w(\alpha_j, \cdot) \}_{j=1}^N$ by $\J_{\Omega,N}^{-1}w$ whenever this does not cause confusion.\\[1mm]
(b) We could as well approximate $\J_\Omega^{-1}$ for points outside $\Omega_H^\Lambda$ by approximating the corresponding integral from~\eqref{eq:JOmegaInverse} and analyze its convergence by the very same analysis we use for $\J_{\Omega,N}^{-1}$, 
\begin{equation}\label{eq:aux714} 
  \left( \J_\Omega^{-1} w \right) \left(x+ \left( \begin{smallmatrix} \Lambda j \\ 0 \end{smallmatrix} \right) \right)
  \approx \left[ \frac{\Lambda}{2\pi} \right]^{1/2} \frac{2\pi}{N\Lambda^*} \sum_{j=1}^N w_h(\alpha_j, x) e^{\i \, \Lambda j \, \alpha_j},
  \quad x \in \Omega_H^\Lambda, \, j \in \Z.
\end{equation}
In particular, the same convergence rate as for $\J_{\Omega,N}^{-1}$ holds on any (fixed) shifted domain $\Omega_H^\Lambda + (\Lambda j,0)^\top$. 
Once one knows the quasiperiodic approximations $w_h(\alpha_j, \cdot)$ on $\Omega_H^\Lambda$, the numerical effort to evaluate the right-hand side of~\eqref{eq:aux714} is of course minimal. 
\end{remark}

In this and the following section we will show error estimates that yield an estimate for the approximation error between $u$ and $u_{N,h}$. 
Due to the two steps of the numerical scheme described by~\eqref{numeric:solution}, the error is due to both numerical approximation of the inverse Bloch transform and numerical approximation of solutions to the quasiperiodic scattering problems by the finite element method.

We first study the error between the exact inverse Bloch transform $\J^{-1}_{\Omega} w$ and the numerical approximation $\J^{-1}_{\Omega,N} w$ for an arbitrary $w \in W^{1,p}_\p(\Wast;H^1_\alpha(\Omega_H^\Lambda))$ with $p \in (1, \infty]$.  
To this end, we define the piecewise linear interpolation operator $I_N$ by 
\begin{equation}
(I_N w)(\alpha,x)=\frac{\alpha-\alpha_{j-1}}{\alpha_j-\alpha_{j-1}}\left[w(\alpha_j,x)-w(\alpha_{j-1},x)\right]+w(\alpha_{j-1},x) \quad \text{for } \alpha\in (\alpha_{j-1},\alpha_j] 
\end{equation}
with $j=1,2,\dots,N$, and $x \in \Omega_H^\Lambda$. 
Note that $I_N$ is well-defined on $W^{1,p}_\p(\Wast;H^1_\alpha(\Omega_H^\Lambda))$ with $p\in (1,\infty]$, because $\Wast$ is a one-dimensional interval and Sobolev's embedding result for functions defined on $\Wast$ with values in $H^1(\Omega_H^\Lambda)$ states that such functions are H\"older continuous with H\"older exponent $1-1/p>0$.  
From the periodicity of $w$ in $\alpha$, it is further easy to see that 
\begin{equation}
(\J^{-1}_{\Omega,N}w)(x)=\left[\frac{\Lambda}{2\pi}\right]^{1/2}\int_{-\pi/\Lambda^*}^{\pi/\Lambda^*}(I_N w)(\alpha,x) \d{\alpha}, \qquad x \in \Omega_H^\Lambda. 
\end{equation}

Before the study of the approximation error, we state a classical Minkowski integral inequality, see~\cite[Theorem 202]{Hardy1988}.
\begin{lemma}
Suppose $(S_1,\mu_1)$ and $(S_2,\mu_2)$ are two measure spaces and $F: \, S_1\times S_2\rightarrow \R$ is measurable. Then the following inequality holds for any $p \geq 1$
\begin{equation}
\left[\int_{S_2}\left|\int_{S_1}F(y,z)\d{\mu_1(y)}\right|^p \d{\mu_2(z)}\right]^{1/p} 
\leq \int_{S_1}\left(\int_{S_2}|F(y,z)|^p \d{\mu_2(z)} \right)^{1/p} \d{\mu_1(y)}.
\label{minkowski}\end{equation}
\end{lemma}

\begin{lemma}
For any function $w \in W^{1,p}_\p(\Wast;H^1_\alpha(\Omega_H^\Lambda))$ with $p \in[1,2)$, the following inequalities hold for $j=1,2$: 
\begin{eqnarray}
  \int_{\Omega_H^\Lambda }\left[\int_{\Wast}|w(\alpha,x)|^p\d{\alpha}\right]^{2/p} \hspace*{-1mm} \d{x}  &\leq& \|w\|_{L^p(\Wast;H^1_\alpha(\Omega_H^\Lambda))}^2, \label{inequ1}\\
  \int_{\Omega_H^\Lambda }\left[\int_{\Wast}\left|\frac{\partial w}{\partial x_j}(\alpha,x)\right|^p \hspace*{-1mm} \d{\alpha}\right]^{2/p} \d{x} &\leq& \|w\|_{L^p(\Wast;H^1_\alpha(\Omega_H^\Lambda))}^2 ,\label{inequ2}\\
  \int_{\Omega_H^\Lambda }\left[\int_{\Wast}\left|\frac{\partial w}{\partial \alpha}(\alpha,x)\right|^p  \hspace*{-1mm} \d{\alpha}\right]^{2/p} \d{x} &\leq& \|w\|_{W^{1,p}_\p(\Wast;H^1_\alpha(\Omega_H^\Lambda))}^2,\label{inequ3}\\
 \int_{\Omega_H^\Lambda }\left[\int_{\Wast}\left|\frac{\partial}{\partial \alpha} \frac{\partial w}{\partial x_j}(\alpha,x)\right|^p \hspace*{-1mm} \d{\alpha}\right]^{2/p} \d{x} &\leq& \|w\|_{W^{1,p}_\p(\Wast;H^1_\alpha(\Omega_H^\Lambda))}^2. \label{inequ4}
\end{eqnarray}
\end{lemma}
\begin{proof}
We prove, as an example, inequality \eqref{inequ1}.  
As $w \in W^{1,p}_\p(\Wast;H^1_\alpha(\Omega_H^\Lambda))$,
\begin{equation*}
\|w\|^p_{W^{1,p}_\p(\Wast;H^1_\alpha(\Omega_H^\Lambda))}=
\int_{\Wast}\left[\|w(\alpha,\cdot)\|^p_{H^1_\alpha(\Omega_H^\Lambda)}+\left\|\frac{\partial}{\partial\alpha}w(\alpha,\cdot)\right\|^p_{H^1_\alpha(\Omega_H^\Lambda)}\right]\d{\alpha}<\infty,
\end{equation*}
such that 
\begin{eqnarray*}
\int_{\Wast}\bigg[\int_{\Omega_H^\Lambda } |w(\alpha,x)|^2 \d{x} \bigg]^{p/2}\d{\alpha} \leq \|w\|^p_{W^{1,p}_\p(\Wast;H^1_\alpha(\Omega_H^\Lambda))}.
\end{eqnarray*}
The inequality \eqref{minkowski} for $y=\alpha$, $z=x$, $F(y,z) = |w(\alpha,x)|^2$, and $p$ replaced by $2/p > 1$ hence implies that  
\begin{eqnarray*}
\left(\int_{\Omega_H^\Lambda }\left[\int_{\Wast}|w(\alpha,x)|^p\d{\alpha}\right]^{2/p} \d{x} \right)^{1/2}
&\leq&
\int_{\Wast}\left[\int_{\Omega_H^\Lambda }|w(\alpha,x)|^{p\cdot 2/p} \d{x} \right]^{p/2}\d{\alpha}\\
&\leq &\|w\|_{W^{1,p}_\p(\Wast;H^1_\alpha(\Omega_H^\Lambda))}.
\end{eqnarray*}
So \eqref{inequ1} is proved. 
The inequalities \eqref{inequ2}, \eqref{inequ3} and \eqref{inequ4} can be proved similarly.
\end{proof}

The bounds from the last lemma allow to prove the following error estimate for the numerical approximation of the inverse Bloch transform.

\begin{theorem}\label{th:approxJ-1}
For $w\in W^{1,p}_\p(\Wast;H^1_\alpha(\Omega_H^\Lambda))$ with $p \in (1,\infty]$, the numerical approximation error between $\J^{-1}_{\Omega,N}w$ and $\J^{-1}_{\Omega}w$ is bounded by
\begin{equation} \label{eq:aux716}
\| (\J^{-1}_{\Omega,N} - \J^{-1}_{\Omega}) w\|_{H^1(\Omega_H^\Lambda)}\leq C N^{-1} \|w(\alpha,x)\|_{W^{1,p}_\p(\Wast;H^1_\alpha(\Omega_H^\Lambda))},
\end{equation}
where $C$ is independent of $N \in\N$.
\label{num:invBloch}\end{theorem}

\begin{proof}
Density of smooth functions in $W^{1,p}_\p(\Wast;H^1_\alpha(\Omega_H^\Lambda))$ implies that it is sufficient to prove~\eqref{eq:aux716} for a function $w$ that is continuously differentiable in $\alpha$. 
For any $t_1,t_2\in \Wast$ with $t_1<t_2$ and $x\in\Omega_H^\Lambda$, we define the conjugate index $q$ to $p$ by $1/q+1/p=1$ and estimate 
\begin{eqnarray*}
|w(t_2,x)-w(t_1,x)|&=&\left|\int_{t_1}^{t_2}\frac{\partial w}{\partial\alpha}(\alpha,x)\d{\alpha}\right|\\&\leq& \left(\int_{t_1}^{t_2}\left|\frac{\partial w}{\partial\alpha}(\alpha,x)\right|^p\d{\alpha}\right)^{1/p}\left(\int_{t_1}^{t_2}1^q \d{\alpha}\right)^{1/q}\\
&\leq & (t_2-t_1)^{1/q}\left\|\frac{\partial w}{\partial\alpha}(\cdot,x)\right\|_{L^p(t_1,t_2)}.
\end{eqnarray*}
With this result, the error of the piecewise linear interpolation of $w$ for $\alpha \in (\alpha_{j-1},\alpha_j]$ is bounded by 
\begin{align*}
 \big|(I_N w)(\alpha,x) &- w(\alpha,x)\big| = \left|\frac{\alpha-\alpha_{j-1}}{\alpha_j-\alpha_{j-1}}\left[w(\alpha_j,x)-w(\alpha_{j-1},x)\right]+w(\alpha_{j-1},x)-w(\alpha,x)\right| \\
&\leq \left|\frac{\alpha-\alpha_{j-1}}{\alpha_j-\alpha_{j-1}}\right|\, \big|w(\alpha_j,x)-w(\alpha_{j-1},x)|+|w(\alpha,x)-w(\alpha_{j-1},x) \big|\\
&\leq \left[\left(\alpha-\alpha_{j-1}\right)(\alpha_j-\alpha_{j-1})^{1/q-1}+(\alpha-\alpha_{j-1})^{1/q}\right]\left\|\frac{\partial}{\partial\alpha}w(\cdot,x)\right\|_{L^p(\alpha_{j-1},\alpha_j)}, 
\end{align*}
where $\alpha_0 := -\pi/\Lambda^*$ and $w(\alpha_0,\cdot) := \exp(\i \,\Lambda \,\alpha) \, w(\alpha_N,\cdot) \, [= \exp(\i \,\Lambda \,\alpha) w(\pi/\Lambda^*,\cdot)]$ to respect the quasiperiodicity of $\alpha \mapsto w(\alpha,\cdot)$. 
Integrating over the whole interval $\Wast = (-\pi/\Lambda^*,\pi/\Lambda^*]$, we find that 
\begin{align*}
&\int_{-\pi/\Lambda^*}^{\pi/\Lambda^*}   \big| (I_N w)(\alpha,x) - w(\alpha,x) \big| \d{\alpha} 
= \sum_{j=1}^N\int_{\alpha_{j-1}}^{\alpha_j} \big|(I_N w)(\alpha,x)-w(\alpha,x)\big| \d{\alpha}\\
& \qquad \leq \sum_{j=1}^N \int_{\alpha_{j-1}}^{\alpha_j} \left[\left(\alpha-\alpha_{j-1}\right)(\alpha_j-\alpha_{j-1})^{1/q-1}+(\alpha-\alpha_{j-1})^{1/q}\right] \hspace*{-1mm} \d{\alpha} \left\|\frac{\partial w}{\partial\alpha}(\cdot,x)\right\|_{L^p(\alpha_{j-1},\alpha_j)}\\
&\qquad \leq\sum_{j=1}^N \left[ \frac12 \left[\frac{2\pi}{N\Lambda^*} \right]^2 + \frac{1}{1+1/q} \right] (\alpha_j-\alpha_{j-1})^{1+1/q} \left\|\frac{\partial w}{\partial\alpha}(\cdot,x)\right\|_{L^p(\alpha_{j-1},\alpha_j)} \\
& \qquad \leq \left[\frac{2\pi^2}{[\Lambda^*]^2} + \frac{1}{2}\right] N^{1-1/p} \left|\frac{2\pi}{N\Lambda^*}\right|^{1+1/q}\left\|\frac{\partial w}{\partial\alpha} (\cdot,x)\right\|_{L^p(\Wast)}
= C N^{-1} \left\|\frac{\partial w}{\partial\alpha} (\cdot,x)\right\|_{L^p(\Wast)},
\end{align*}
where we used that $a_1^{1/p} + \dots + a_N^{1/p} \leq N^{1-1/p} (a_1+ \dots+a_N)^{1/p}$ for $a_j= \|(\partial w/\partial\alpha)(\cdot,x) \|_{L^p(\alpha_{j-1},\alpha_j)}^p$. 
By~\eqref{inequ1}, the squared $L^2$-norm in $x\in\Omega_H^\Lambda$ is hence bounded by
\begin{eqnarray*}
\left\|\int_{-\pi/\Lambda^*}^{\pi/\Lambda^*} \big| (I_Nw)(\alpha,\cdot)-w(\alpha,\cdot)\big| \d{\alpha}\right\|_{L^2(\Omega_H^\Lambda)}^2 
&\leq&
C^2 N^{-2} \int_{\Omega_H^\Lambda }\left\|\frac{\partial}{\partial\alpha}w(\alpha,x)\right\|^2_{L^p(\Wast)} \d{x} \\
&\leq&  C^2 N^{-2} \| w \|_{W^{1,p}_\p(\Wast;H^1_\alpha(\Omega_H^\Lambda))}^2.
\end{eqnarray*}
With similar arguments, one can also show the estimate
\begin{eqnarray*}
\bigg\| \int_{-\pi/\Lambda^*}^{\pi/\Lambda^*} \bigg| \frac{\partial}{\partial x_j} [(I_N w)(\alpha,\cdot)-w(\alpha,\cdot)] \bigg| \d{\alpha} \bigg\|_{L^2(\Omega_H^\Lambda)}
\leq  C N^{-1} \| w \|_{W^{1,p}_\p(\Wast;H^1_\alpha(\Omega_H^\Lambda))},
\end{eqnarray*}
which finally implies the claimed bound. 
% \begin{eqnarray*}
% \left\|\J^{-1}_{\Omega}w-\J^{-1}_{\Omega,N}w\right\|_{H^1\left(\Omega_H^\Lambda \right)}
% &=& \left[\frac{\Lambda}{2\pi}\right]^{1/2} \bigg\|\int_{-\pi/\Lambda^*}^{\pi/\Lambda^*}[(I_N w)(\alpha,\cdot)-w(\alpha,\cdot)]\d{\alpha}\bigg\|_{H^1\left(\Omega_H^\Lambda \right)}\\
% &\leq& C\left|\frac{2\pi}{N\Lambda^*}\right|^{1+1/q}\|w(\alpha,x)\|_{W^{1,p}_\p(\Wast;H^1_\alpha(\Omega_H^\Lambda))}.
% \end{eqnarray*}
\end{proof}
Note that the estimate of the latter theorem is optimal, as it bounds, roughly speaking, $L^1$-errors of interpolation operators applied to functions that are once weakly differentiable in $\alpha$ by $N^{-1}$ times the norm of the interpolated function. 
As we never exploited the particular structure of $\widetilde{H}^1_\alpha(\Omega_H^\Lambda)$, Theorem~\ref{th:approxJ-1} holds analogously for functions in $W^{1,p}_\p(\Wast;X)$ with arbitrary Banach spaces $X$. 

\section{Estimates for Discrete Periodic Scattering Problems and Convergence of the Numerical Scheme}\label{eq:errFemTotal}
After establishing bounds for the discretized inverse Bloch transform in the last section, we now aim to first prove an error estimate for a conforming element approximation of the solution $w = w(\alpha, \cdot)$ of the quasiperiodic scattering problem~\eqref{eq:heAlpha}.  
This allows in a second step to prove convergence of the numerical scheme~\eqref{numeric:solution} for the approximation of the solution $u$ to the scattering problem~\eqref{eq:varFormHEScal} in $\Omega_H$. 

To establish convergence of finite element approximations to~\eqref{eq:heAlpha}, we use standard duality arguments, coupled with regularity and uniqueness results for the adjoint problem to~\eqref{eq:heAlpha}. 
Similar results for a related transmission problem can be found in~\cite{Bao1995}. 
Recall that the solution $w = w(\alpha, \cdot) \in \widetilde{H}^1_\alpha(\Omega_H^\Lambda)$ to~\eqref{eq:heAlpha} solves, by definition, 
\begin{equation}\label{eq:heAlpha2}
  a_\alpha(w,v)
  = \int_{\Gamma_H^\Lambda} f_\alpha \overline{v} \dS 
  \qquad \text{for } f_\alpha = \left[ \frac{\partial }{\partial \nu} \J_\Omega u^i (\alpha,\cdot)- T^+ [\J_\Omega u^i (\alpha,\cdot) ]\big|_{\Gamma_H^\Lambda} \right], %\in H^{-1/2}_\alpha(\Gamma_H^\Lambda)
\end{equation}
and for all $v \in \widetilde{H}^1_\alpha(\Omega_H^\Lambda)$. 
Note that whenever one approximates~\eqref{eq:heAlpha2} in finite-dimensional subspaces of $\widetilde{H}^1_\alpha(\Omega_H^\Lambda)$, the periodic Dirichlet-to-Neumann operator $T^+_\alpha$ naturally has to be truncated. 
We omit in the following to tackle this additional truncation error and merely consider the discretization error due to the finite-dimensional variational space, noting however that the truncation error is analyzed in, e.g.,~\cite{Bao1995, George2011}.
% 
% When solving the quasiperiodic scattering problems, we use a truancated Dirichlet-to-Neumann operator $\tilde{T}^+_{\alpha_j}$ to approximate $T^+_{\alpha_j}$. From the analysis in the \cite{George2011}, for suitable choice of the number of truncated terms, the error between $w_h(\alpha_j,x)$ and $w(\alpha_j,x)$ has the estimate
% \begin{equation}
% \|w_h(\alpha_j,x)-w(\alpha_j,x)\|
% \end{equation}

Assume from now on that $\Omega$ is a domain of class $C^{1,1}$, that is, $\Gamma$ is graph of a function in $C^{1,1}(\R)$, and that $\Omega_H^\Lambda$ is (exactly) covered by a family of regular curved and quasi-uniform meshes with mesh widths $0<h\leq h_0$, see~\cite{Saute2007,Brenn1994}. 
These meshes yield a family of discrete spaces $V_{h} \subset \widetilde{H}^1(\Omega_H^\Lambda))$ of piecewise linear and globally continuous functions $v_h$. 
For $\alpha \in \Wast$ we additionally defined a quasiperiodic subspace $V_h^\alpha \subset \widetilde{H}^1_\alpha(\Omega_H^\Lambda))$ containing all elements $v_h\in V_h$ that satisfy the quasiperiodicity condition 
\[
  \left. (v_h) \right|_{\{x_1=\pi/\Lambda \}} = e^{\i \, \Lambda \, \alpha} \,  \left. (v_h) \right|_{\{x_1=-\pi/\Lambda \}}.
\] 
(This obviously requires the $x_2$-coordinates of mesh points on $\{x_1=\pm\pi/\Lambda \}$ to pairwise match each other in order to allow for non-constant boundary values and simple implementation.)
By construction, elements in $V_h^\alpha$ can then be $\alpha$-quasiperiodically extended to quasiperiodic functions in $H^1_\loc(\Omega_H)$. 

\begin{theorem}\label{th:convRateFEM} 
Suppose $\Omega_H$ is a domain of class $C^{1,1}$, let $\alpha \in \Wast$, define $f_\alpha$ as in~\eqref{eq:heAlpha2} for some incident field $u^i$ that solves the Helmholtz equation in $\Omega$, and assume that $f_\alpha \in H^{1/2}_\alpha(\Gamma_H^\Lambda)$. 
Then the solution $w$ to~\eqref{eq:heAlpha2} belongs to $H^2(\Omega_H^\Lambda)$. 
Further, if $h \in (0,h_0]$ and $w_h \in V_h^\alpha$ solves 
\begin{equation}%\label{eq:heAlpha3}
  a_\alpha(w_h,v_h)
  = \int_{\Gamma_H^\Lambda} f_\alpha \overline{v_h} \dS \quad \text{for all $v_h \in V_h^\alpha$}, 
\end{equation}
then $\| w_h - w \|_{L^2(\Omega_H^\Lambda)} + h \| \nabla(w_h - w) \|_{L^2(\Omega_H^\Lambda)^3}  \leq C h^2 \| f_\alpha \|_{H^{1/2}_\alpha(\Gamma_H^\Lambda)}$ for $C$ independent of $\alpha \in \Wast$.  
\end{theorem}
\begin{remark}
It is crucial that the assumption of the latter theorem guarantee that $w$ is $H^2$-regular, such that, e.g., convexity assumptions for the domain that yield the same regularity would also be applicable. 
\end{remark}
\begin{proof}
(1) 
As $\Omega_H^\Lambda$ is not a domain of class $C^{1,1}$, we cannot directly apply regularity results to the solution $w$. 
Consider instead the extension of $w \in \widetilde{H}^1_\alpha(\Omega_H^\Lambda)$ by quasiperiodicity to $\Omega_H$, and then by~\eqref{eq:URCAlpha} to a weak solution of the Helmholtz equation in $H^1_\loc(\Omega)$.   
As $\Omega$ is by assumption a domain of class $C^{1,1}$, standard elliptic regularity results imply that the extension of $w$ belongs to $H^2_\loc(\Omega)$, such that $w$ itself belongs to $\widetilde{H}^1_\alpha(\Omega_H^\Lambda) \cap H^2(\Omega_H^\Lambda)$. 
Moreover, there is $C>0$ independent of $\alpha $, $w$ or $f_\alpha$ such that $\| w \|_{H^2(\Omega_H^\Lambda)} \leq C \| f_\alpha \|_{H^{1/2}_\alpha(\Gamma)}$. 

(2) 
As the variational formulation~\eqref{eq:heAlpha2} is of Fredholm type with index zero, and since discretization space $V_h^\alpha \subset \widetilde{H}^1_\alpha(\Omega_H^\Lambda)$ is conforming, the standard quasi-optimal convergence result for $w_h$ is that 
\begin{equation}\label{eq:quasiOpti}
  \| w - w_h \|_{\widetilde{H}^1_\alpha(\Omega_H^\Lambda)} \leq C \inf_{v_h \in S_h} \| w - v_h \|_{\widetilde{H}^1_\alpha(\Omega_H^\Lambda)}, 
\end{equation}
see~\cite{Saute2007}. 
As in~\cite[Section 4.7]{Brenn1994} one additionally shows that the space quasiperiodic discrete space $V_h^\alpha$ containing piecewise linear and globally continuous elements satisfies  
\begin{equation}\label{eq:approx}
  \inf_{v_h \in V_h^\alpha} \| w - v_h \|_{\widetilde{H}^1_\alpha(\Omega_H^\Lambda)} 
  \leq C h \| w \|_{H^2(\Omega_H^\Lambda)}.
\end{equation}
(This estimate also holds for other types of finite elements.)
By~\eqref{eq:quasiOpti}, this implies that 
\begin{equation}\label{eq:aux641} 
  \| w - w_h \|_{\widetilde{H}^1_\alpha(\Omega_H^\Lambda)} \leq C h \| w  \|_{H^2(\Omega_H^\Lambda)} 
  \leq C h \| f_\alpha \|_{H^{1/2}_\alpha(\Gamma)}. 
\end{equation} 
Again, the arising constant $C$ can be chosen uniformly in $\alpha \in \Wast$. 

(3) 
The adjoint problem to find $z \in \widetilde{H}^1_\alpha(\Omega_H^\Lambda)$ solving $a_\alpha(v,z) = (v,\phi)_{L^2(\Omega_H^\Lambda)}$ for all $v \in \widetilde{H}^1_\alpha(\Omega_H^\Lambda)$ for some arbitrary $\phi \in \widetilde{H}^1_\alpha(\Omega_H^\Lambda)$ is obviously also Fredholm of index zero. 
As we prove next, uniqueness of solution for this problem follows from uniqueness of solution for the original variational problem~\eqref{eq:heAlpha2}, such that the adjoint problem is also uniquely solvable for all right-hand sides. 

Assume that $z \in \widetilde{H}^1_\alpha(\Omega_H^\Lambda)$ satisfies $a_\alpha(v,z) = 0$ for all $v \in \widetilde{H}^1_\alpha(\Omega_H^\Lambda)$ and denote the Fourier coefficients of $z$ and $v$ on $\Gamma_H$ as $\hat{z}(j)$ and $\hat{v}(j)$, respectively, for $j\in \Z$. 
Choosing $v=z$, we directly obtain from~\eqref{eq:perioDtN} as in~\eqref{eq:aux414} that 
\[ 
  \Im a_\alpha(z,z) = \sum_{j\in\Z} \Im (\i\beta(j)(k,\alpha)) \, |\hat{z}(j)|^2
  = \sum_{j: \, |\Lambda^* j+\alpha|<k} \sqrt{k^2 - |\Lambda^* j + \alpha|^2} \, |\hat{z}(j)|^2,
\]
such that all coefficients $\hat{z}(j)$ corresponding to propagating modes have to vanish. 
The solution $z$ to the homogeneous problem
\begin{align*}
  0 = a_\alpha(v,z) 
  & = \int_{\Omega_H^\Lambda} \left[ \nabla v \cdot \nabla \overline{z} - k^2 v \overline{z} \right] \d{x} 
  -\sum_{j: \, |\Lambda^* j+\alpha| \geq k} \sqrt{|\Lambda^* j + \alpha|^2-k^2} \, \hat{v}(j) \overline{\hat{z}(j)} \\
  & = \overline{a_\alpha(z,v)} \qquad \text{for all } v \in \widetilde{H}^1_\alpha(\Omega_H^\Lambda)
\end{align*}
is hence also a solution to the original variational formulation $a_\alpha(z,v) = 0$ for all $v \in \widetilde{H}^1_\alpha(\Omega_H^\Lambda)$. 
As we already know from Theorem~\ref{th:uniqueAlpha} that this problem is uniquely solvable, we conclude that $z$ vanishes. 

(4) 
The arguments proving the regularity estimate for $w$ from part (1) of this proof directly transfers to the adjoint solution $z$ defined in (3), such that there is $C>0$ independent of $\alpha \in \Wast$ such that $\| z \|_{H^2(\Omega_H^\Lambda)} \leq C \| \phi \|_{L^2(\Omega_H^\Lambda)}$.  
Exploiting Galerkin orthogonality, i.e., $a_\alpha(w-w_h, v_h) =0$ for all $v_h\in V_h^\alpha$, this shows that 
\begin{align*}
 (w-w_h,\phi)_{L^2(\Omega_H^\Lambda)}
 & = a_\alpha(w-w_h, z) = a_\alpha(w-w_h, z-z_h) \\
 & \leq C \| w-w_h \|_{\widetilde{H}^1_\alpha(\Omega_H^\Lambda)} \| z-z_h \|_{\widetilde{H}^1_\alpha(\Omega_H^\Lambda)} \quad \text{for all $z_h\in V_h^\alpha$}.
\end{align*}
Choosing $z_h$ to be the orthogonal projection of $z$ onto $V_h^\alpha$ hence yields by~\eqref{eq:approx} that 
\[
  \| z-z_h \|_{\widetilde{H}^1_\alpha(\Omega_H^\Lambda)} \leq C h \| z \|_{H^2(\Omega_H^\Lambda)} 
  \leq C h \| \phi \|_{L^2(\Omega_H^\Lambda)}.  
\]
Thus, for all $\phi \in \widetilde{H}^1_\alpha(\Omega_H^\Lambda)$ with $L^2$-norm equal to one, we get that $(w-w_h,\phi)_{L^2(\Omega_H^\Lambda)} \leq Ch \| w-w_h \|_{\widetilde{H}^1_\alpha(\Omega_H^\Lambda)}$. 
Density of $\widetilde{H}^1_\alpha(\Omega_H^\Lambda)$ in $L^2(\Omega_H^\Lambda)$ finally shows by taking the supremum over all such $\phi$ that $\| w-w_h \|_{L^2(\Omega_H^\Lambda)} \leq Ch \| w-w_h \|_{\widetilde{H}^1_\alpha(\Omega_H^\Lambda)}$.
Together with~\eqref{eq:aux641}, this bound implies the claim. 
\end{proof}

The error estimates of the numerical approximation of the inverse Bloch transform from Theorem~\ref{th:approxJ-1} together with the error estimates for the finite element approximations from Theorem~\ref{th:convRateFEM} now allow to conclude for an error estimate of the numerical scheme~\eqref{numeric:solution}. 

\begin{theorem} \label{th:main}
Suppose the incident field $u^i \in H^2(\Omega_H)$ satisfies  
\begin{equation}\label{eq:assSource2}
%  \J_\Omega u^i \in W^{1,p}_\p(\Wast; H^2_\alpha(\Omega_H^\Lambda))
%  \quad \text{and} \quad 
  R(u^i) := \sup_{\alpha \in \Wast} \| \J_\Omega u^i(\alpha, \cdot) \|_{H^2_\alpha(\Omega_H^\Lambda)} < \infty 
  \qquad \text{for some $p\in(1,2)$. }
\end{equation}
Then the error between the numerical approximation $u_{N,h}$ from~\eqref{numeric:solution} and the exact solution $u \in H^1(\Omega_H)$ to~\eqref{eq:varFormHEScal} is bounded by
\begin{align*}
\|u_{N,h}-u\|_{L^2(\Omega_H^\Lambda)}\leq C R(u^i)  \left[h^2 + N^{-1}  \right]
\quad \text{and} \quad
\|u_{N,h}-u\|_{H^1(\Omega_H^\Lambda)}\leq C R(u^i) \left[ h + N^{-1} \right].
\end{align*}
\end{theorem}

\begin{remark}
(a) For notational simplicity, the latter theorem merely states approximation results between $u_{N,h}$ and $u$ on $\Omega_H^\Lambda$. 
If one aims to approximate $u$ on regions in $\Omega_H \setminus \Omega_H^\Lambda$, then one first needs to extend the finite element approximations $w_h(\alpha_j,\cdot)$ into these regions to apply the discrete inverse Bloch transform from~\eqref{numeric:inverseBloch} afterwards. 
As quasiperiodic extensions can be computed up to machine precision, the additional numerical error caused by this procedure is marginal.\\[1mm]
(b) Regularity results for general elliptic equations with constant coefficients imply that any twice differentiable solution $u^i$ to the Helmholtz equation in $\Omega_H$ such that $ \sup_{\alpha \in \Wast} \| \J_\Omega u^i(\alpha, \cdot) \|_{H^1_\alpha(\Omega_H^\Lambda)}$ is finite also satisfies that $R(u^i)$ is finite. 
This applies in particular for the two classes of incident fields from Remark~\ref{rem:2}.
\end{remark}

\begin{proof}
Let $X$ be either $L^2(\Omega^\Lambda_H)$ or $\widetilde{H}^1(\Omega^\Lambda _H)$ and denote by $w=\J_\Omega u$ the Bloch transform of $u \in \widetilde{H}^1(\Omega)$ and by $w_h(\alpha,\cdot)$ the solution to~\eqref{eq:heAlpha2} for arbitrary $\alpha\in\Wast$. 
Then 
\begin{align} \label{eq:aux903}
\|u_{N,h}-u\|_X&=\|J^{-1}_{\Omega,N}w_h-J^{-1}_{\Omega}w\|_X  \leq  \|J^{-1}_{\Omega,N}(w_h-w)\|_X+\|(J^{-1}_{\Omega,N}-J^{-1}_{\Omega})w\|_X.
\end{align}
Note that Theorem~\ref{th:exiSolScalPerio}(b) states that $w = \J_\Omega u$ belongs to $W^{1,p}_\p(\Wast; \widetilde{H}^1_\alpha(\Omega_H^\Lambda))$.
The last term in~\eqref{eq:aux903} hence can be bounded due to Theorem~\ref{numeric:inverseBloch} by
\begin{equation} \label{eq:aux916}
  \|(J^{-1}_{\Omega,N}-J^{-1}_{\Omega})w\|_{H^1(\Omega_H^\Lambda)}
  \leq C N^{-1} \|w\|_{W^{1,p}_\p(\Wast;H^1_\alpha(\Omega_H^\Lambda))}.
\end{equation}
Concerning the second-to-last term in~\eqref{eq:aux903}, we already have an estimate for $w_h(\alpha_j,\cdot) - w(\alpha_j,\cdot)$ due to Theorem~\ref{th:convRateFEM}, such that it remains to estimate the operator norm of $J^{-1}_{\Omega,N}$ on $X$. 
To this end, note that this term can be bounded via the definition of $J^{-1}_{\Omega,N}$ by 
\begin{align*}
\|J^{-1}_{\Omega,N}(w_h-w)\|_X &= \left\|\left[\frac{\Lambda}{2\pi}\right]^{1/2}\frac{2\pi}{N\Lambda^*}\sum_{j=1}^N\left(w_h(\alpha_j,\cdot)-w(\alpha_j,\cdot)\right)\right\|_X \\
&\leq \left[\frac{\Lambda}{2\pi}\right]^{1/2}\frac{2\pi}{N\Lambda^*}\sum_{j=1}^N\|w_h(\alpha_j,\cdot)-w(\alpha_j,\cdot)\|_X
 \leq C \frac{h^\ell}{N}\sum_{j=1}^N \| f_{\alpha_j} \|_{H_{\alpha}^{1/2}(\Gamma_H)}, 
\end{align*}
where $\ell=2$ if $X=L^2(\Omega^\Lambda_H)$ and $\ell=1$ if $X=\widetilde{H}^1(\Omega^\Lambda _H)$. 
By the definition of $f_\alpha$ in \eqref{eq:heAlpha2}, the trace bound $\| (\partial / \partial \nu) J_{\Omega} u^i (\alpha,\cdot) \|_{H^{1/2}_\alpha(\Gamma_H^\Lambda)} \leq C \| J_{\Omega} u^i (\alpha,\cdot) \|_{H^2_\alpha(\Gamma_H^\Lambda)}$ (c.f.~the proof of Theorem~\ref{th:exiSolScalPerio}(a)), the uniform boundedness of the operators $T^+_\alpha$ between $H^{\pm 1/2}_\alpha(\Gamma_h^\Lambda)$ (see the proof of Theorem~\ref{th:exiSolScalPerio}(b)), and the trace theorem, we infer that 
\begin{align*}
\| f_\alpha \|_{H^{1/2}_\alpha(\Gamma_h^\Lambda)}
& \leq C \left\| \frac{\partial}{\partial\nu} J_{\Omega} u^i (\alpha,\cdot) \right\|_{H^{1/2}_\alpha(\Gamma_H)} + \| J_{\Gamma_H} u^i (\alpha,\cdot) \|_{H^{3/2}_\alpha(\Gamma_H^\Lambda)}  \\ 
& \leq C \| J_\Omega u^i (\alpha,\cdot) \|_{H^2_\alpha(\Omega_H^\Lambda)} 
\leq \sup_{\alpha \in \Wast} \| \J_\Omega u^i(\alpha, \cdot) \|_{H^2_\alpha(\Omega_H^\Lambda)} = R(u^i) <\infty,
\end{align*}
which is a finite expression due to~\eqref{eq:assSource2}. 
Consequently, $\|J^{-1}_{\Omega,N}(w_h-w)\|_X \leq C R(u^i) h^\ell$. 
Combining the latter bound with~\eqref{eq:aux916} and~\eqref{eq:aux903}, we arrive at $\|u_{N,h}-u\|_X \leq C \big[R(u^i) h^\ell + N^{-1} \| \J_\Omega u \|_{W^{1,p}_\p(\Wast;H^1_\alpha(\Omega_H^\Lambda))} \big]$. 
We finally bound $\| \J_\Omega u \|_{W^{1,p}_\p(\Wast;H^1_\alpha(\Omega_H^\Lambda))}$ using Theorem~\ref{th:exiSolScalPerio}(b) by a fixed constant times $\sup_{\alpha\in\Wast} \| \J_\Omega u^i(\alpha,\cdot) \|_{H^1_\alpha(\Omega_H^\Lambda)}$, which yields the claimed error estimates. 
The proof is finished.
\end{proof}

\section{Numerical Examples}\label{se:num}
In this section, we will show some numerical results for our computational method that in particular confirm the theoretical convergence rates. 
To this end, we consider incident fields in form of half-space point sources from~\eqref{eq:greenFunc}, i.e., 
\begin{equation*}
u^i(x) := G(x,y) = \frac{\i}{4} \left[ H^{(1)}_0(k|x-y|) - H^{(1)}_0(k|x-y'|) \right], \quad x \not = y \in \R^2,
\end{equation*}
with $y$ located below the $\Lambda$-periodic surface $\Gamma$ with $y_2>0$ and $y'=(y_1,-y_2)^\top$. In our numerical examples, we fix the period $\Lambda$ to be $2\pi$. From \eqref{eq:BlochZeta}, the Bloch transform of the incident field is defined as
\begin{eqnarray*}
(J_{\R}u^i)(\alpha,x)
= \sum_{j\in\Z}u^i (x_1+2\pi j,x_2) e^{-2\pi\i \, j \, \alpha}, 
\quad x +2\pi \left( \begin{smallmatrix} j\\ 0 \end{smallmatrix} \right) \not = y \in \R^2.
\end{eqnarray*}
Recall the definition of the $\alpha$-quasiperiodic Green's function of the Helmholtz equation in two dimensions,
\begin{equation*}
\Phi_\alpha(x,y)=\frac{\i}{4}\sum_{j\in\Z}H_0^{(1)}\left(k\sqrt{(x_1-2\pi j-y_1)^2+(x_2-y_2)^2} \right) e^{2\pi\i \, j \, \alpha}, \quad 
x_2 \not = y_2. 
\end{equation*}
It is easy to show that 
\begin{equation*}
(J_{\R}u^i)(\alpha,x) = \Phi_\alpha(x,y) - \Phi_\alpha(x,y'), \quad 
x_2 \not = \pm y_2.  
\end{equation*}
Let us now, in analogy to~\eqref{eq:URCAlpha}, define numbers $\alpha(j) = \Lambda^*j + \alpha = j + \alpha$ (not to be mixed up with the discretization points $\alpha_j$!) that fit the $\beta(j)$ from~\eqref{eq:URCAlpha},
\begin{equation*}
\alpha(j) = j+\alpha,\quad \beta(j)=\begin{cases}
\sqrt{k^2-\alpha(j)^2}, & \text{if }|\alpha(j)|\leq k,\\
\i\sqrt{\alpha(j)^2-k^2}, & \text{if }|\alpha(j)|>k,
\end{cases}
\qquad \text{for } j \in\Z. 
\end{equation*}
For $\alpha$ such that $\beta(j)\neq 0$ for any $j\in\Z$, the $\alpha$-quasiperiodic Green's function has the basic eigenfunction expansion 
\begin{equation*}
\Phi_\alpha(x,y)=\frac{\i}{4\pi}\sum_{j\in\Z}\frac{1}{\beta(j)} e^{\i\alpha(j)(x_1-y_1)+\i\beta(j)|x_2-y_2|}, \quad 
x_2 \not = \pm y_2.
\end{equation*}
As $y$ is located below the surface $\Gamma$, then in the domain $\{ x\in\Omega: \, x_2 > y_2 \}$ and in particular on $\Gamma_H$ there holds that 
\begin{equation*}
\Phi_\alpha(x,y)=\frac{\i}{4\pi}\sum_{j\in\Z}\frac{1}{\beta(j)} e^{\i\alpha(j)(x_1-y_1)+\i\beta(j)(x_2-y_2)}, \quad x_2 > y_2.
\end{equation*}
Similarly, we also note that
\begin{equation*}
\Phi_\alpha(x,y')=\frac{\i}{4\pi}\sum_{j\in\Z}\frac{1}{\beta(j)} e^{\i\alpha(j)(x_1-y_1)+\i\beta(j)(x_2+y_2)}, \quad x_2 > 0.
\end{equation*}
Thus, $J_{\R}u^i$ can also be represented in the form
\begin{eqnarray}
(J_{\R}u^i)(\alpha,x)&=&\frac{\i}{4\pi}\sum_{j\in\Z}\frac{1}{\beta(j)} e^{\i\alpha(j)(x_1-y_1)+\i\beta(j)x_2}\left[ e^{-\i\beta(j) y_2} - e^{\i\beta(j) y_2} \right] \nonumber \\
&=&\frac{1}{2\pi}\sum_{j\in\Z} e^{\i\alpha(j)(x_1-y_1)+\i\beta(j)x_2} \, \mathrm{sinc}(\beta(j)y_2) \, y_2, \label{eq:repUi}
\end{eqnarray}
where $\rm{sinc}$ is the (smooth) function defined by $\rm{sinc}(t)=\sin(t)/t$ for $t\neq 0$ and $\rm{sinc}(0) = 1$. 
The formula in~\eqref{eq:repUi} allows to evaluate the Bloch transform of the incident field $u^i = G(\cdot,y)$ in case that there is $j\in\N$ such that $\beta(j)=0$. 
\begin{remark}
Formula~\eqref{eq:repUi} in particular allows to evaluate the $\alpha$-quasiperiodic Green's function in case that one of the square roots $\beta(j)$ vanishes. 
This is a nice feature if one aims to use boundary integral equations to solve quasiperiodic scattering problems. 
\end{remark}
\begin{figure}[tttttt!!!b]
\centering
\begin{tabular}{c c c}
\includegraphics[width=0.3\textwidth]{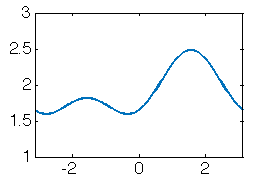} 
& \includegraphics[width=0.3\textwidth]{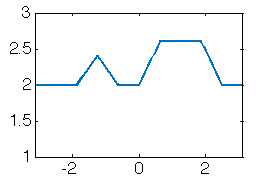} 
& \includegraphics[width=0.3\textwidth]{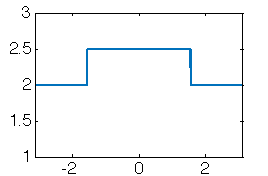}\\[-0cm]
(a) & (b) & (c) 
\end{tabular}%
\begin{picture}(0,0)
  \put(-12,-0.3){$f_1$}
  \put(-6.85,-0.3){$f_2$}
  \put(-1.75,-0.3){$f_3$}
\end{picture}%
\caption{(a)-(c): The three surfaces $\Gamma_{1,2,3}$ defined by the functions $f_{1,2,3}$.}
\end{figure}
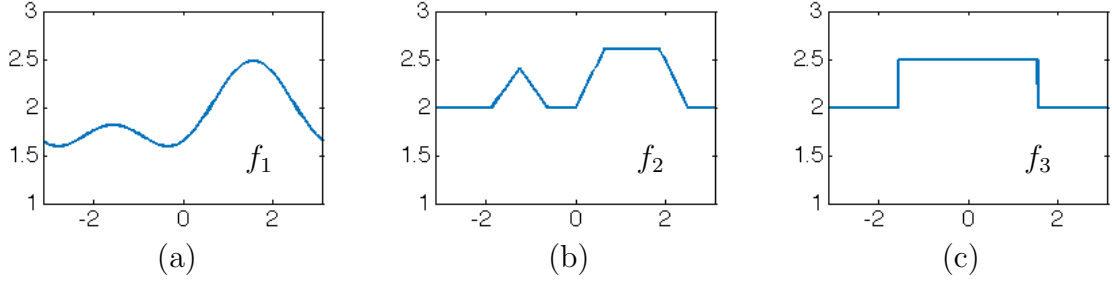

In the remainder of this section, we present numerical experiments for three different periodic surfaces. The first surface, denoted by $\Gamma_1$, is a smooth surface defined by
\begin{equation*}
f_1(t)=\frac{\sin(t)}{3}-\frac{\cos(2t)}{4}+1.9  \qquad \text{ for } t\in\R.
\end{equation*}
The second surface, denoted by $\Gamma_2$, is a graph of the Lipschitz function
\begin{equation*}
f_2(t)=\begin{cases}
2(t/\pi+0.6)+2, & t\in[-0.6\pi,-0.4\pi],\\
2(-t/\pi-0.2)+2, & t\in (-0.4\pi,-0.2\pi],\\
3t/\pi+2, & t\in[0,0.2\pi],\\
2.6, & t\in(0.2\pi,0.6\pi),\\
3(-t/\pi+0.8)+2, & t\in [0.6\pi,0.8\pi],\\
2, & \text{otherwise,}
\end{cases}
\qquad \text{ for } t\in\R.
\end{equation*}
The third surface, denoted by $\Gamma_3$, is defined by a piecewise constant function
\begin{equation*}
f_3(t)=\begin{cases}
2.5, & t\in[-0.5\pi,0.5\pi],\\
2, & \text{otherwise,}
\end{cases}
 \qquad \text{ for } t\in\R.
\end{equation*}
For each surface, we chose $H=3$ and computed numerical solutions $u_{N,h}$ defined in~\eqref{numeric:solution} on $\Omega_3^{2\pi}$ for incident fields $u^i = G(\cdot,y)$ with different locations of the source point $y$ and two wave numbers $k=1$ and $k=10$ (such that the wave lengths equal 6.28 and 0.628, respectively).  
The sum of the variational form $a_\alpha$ implementing the operator $T^+_\alpha$ is truncated at $M=80$. 
As the exact solution $u$ to the scattering problem equals the explicitly known function $G(\cdot, y)$, we can then compute (up to quadrature errors) the relative $L^2$-error $\| u_{N,h} - u \|_{L^2(\Omega_3^{2\pi})} / \| u \|_{L^2(\Omega_3^{2\pi})}$. 

In Examples 1 and 2, the periodic surface is $\Gamma_1$, the source points are $y=(-1,0.4)^\top$ and $y=(0.5,0.2)^\top$, and the wave number equals $k=1$ and $k=10$, respectively. 
Examples 3 and 4 share the setting of Examples 1 and 2 for the periodic surface $\Gamma_2$ instead of $\Gamma_1$.
Examples 5 and 6 share the setting of Examples 1 and 2 as well but involve $\Gamma_3$.

The numerical examples are computed for $N=20,40,80,160,320$ quasiperiodicities and and mesh widths $h=0.16,0.08,0.04,0.02,0.01$. The following tables show the relative errors in the $L^2(\Omega_3^{2\pi})$-norm. 
For surface $\Gamma_1$, see Tables \ref{surf1k1} and~\ref{surf1k10}, for surface $\Gamma_2$ see Tables \ref{surf2k1} and~\ref{surf2k10}, and for surface $\Gamma_3$ see Tables \ref{surf3k1} and~\ref{surf3k10}.

\begin{table}[htb]
\centering
\caption{Relative $L^2$-errors for Example 1 (surface $\Gamma_1$, source at $y=(-1,0.4)^\top$, $k=1$).}\label{surf1k1}
\begin{tabular}
{|p{1.8cm}<{\centering}||p{2cm}<{\centering}|p{2cm}<{\centering}
 |p{2cm}<{\centering}|p{2cm}<{\centering}|p{2cm}<{\centering}|}
\hline
  & $h=0.16$ & $h=0.08$ & $h=0.04$ & $h=0.02$ & $h=0.01$\\
\hline
\hline
$N=20$&$1.65$E$-02$&$1.59$E$-02$&$1.59$E$-02$&$1.59$E$-02$&$1.58$E$-02$\\
\hline
$N=40$&$6.09$E$-03$&$5.70$E$-03$&$5.62$E$-03$&$5.61$E$-03$&$5.60$E$-03$\\
\hline
$N=80$&$2.68$E$-03$&$2.10$E$-03$&$2.00$E$-03$&$1.99$E$-03$&$1.98$E$-03$\\
\hline
$N=160$&$1.69$E$-03$&$8.61$E$-04$&$7.27$E$-04$&$7.06$E$-04$&$7.01$E$-04$\\
\hline
$N=320$&$1.46$E$-03$&$4.84$E$-04$&$2.83$E$-04$&$2.54$E$-04$&$2.49$E$-04$\\
\hline
\end{tabular}
\end{table}

\begin{table}[htb]
\centering
\caption{Relative $L^2$-errors for Example 2 (surface $\Gamma_1$, source at $y=(0.5,0.2)^\top$, $k=10$).}\label{surf1k10}
\begin{tabular}
{|p{1.8cm}<{\centering}||p{2cm}<{\centering}|p{2cm}<{\centering}
 |p{2cm}<{\centering}|p{2cm}<{\centering}|p{2cm}<{\centering}|}
\hline
  & $h=0.08$ & $h=0.04$ & $h=0.02$ & $h=0.01$\\
\hline
\hline
$N=20$&$1.99$E$-01$&$5.64$E$-02$&$3.09$E$-02$&$3.03$E$-02$\\
\hline
$N=40$&$1.99$E$-01$&$5.30$E$-02$&$1.63$E$-02$&$1.11$E$-02$\\
\hline
$N=80$&$1.99$E$-01$&$5.28$E$-02$&$1.36$E$-02$&$4.93$E$-03$\\
\hline
$N=160$&$1.99$E$-01$&$5.28$E$-02$&$1.33$E$-02$&$3.54$E$-03$\\
\hline
$N=320$&$1.99$E$-01$&$5.29$E$-02$&$1.33$E$-02$&$3.36$E$-03$\\
\hline
\end{tabular}
\end{table}

\begin{table}[ht]
\centering
\caption{Relative $L^2$-errors for Example 3 (surface $\Gamma_2$, source at $y=(-1,0.4)^\top$, $k=1$).}\label{surf2k1}
\begin{tabular}
{|p{1.8cm}<{\centering}||p{2cm}<{\centering}|p{2cm}<{\centering}
 |p{2cm}<{\centering}|p{2cm}<{\centering}|p{2cm}<{\centering}|}
\hline
  & $h=0.16$ & $h=0.08$ & $h=0.04$ & $h=0.02$ & $h=0.01$\\
\hline
\hline
$N=20$&$1.78$E$-02$&$1.77$E$-02$&$1.77$E$-02$&$1.77$E$-02$&$1.77$E$-02$\\
\hline
$N=40$&$6.46$E$-03$&$6.30$E$-03$&$6.26$E$-03$&$6.25$E$-03$&$6.25$E$-03$\\
\hline
$N=80$&$2.55$E$-03$&$2.27$E$-03$&$2.22$E$-03$&$2.21$E$-03$&$2.21$E$-03$\\
\hline
$N=160$&$1.36$E$-03$&$8.64$E$-04$&$7.96$E$-04$&$7.85$E$-04$&$7.82$E$-04$\\
\hline
$N=320$&$1.07$E$-03$&$4.10$E$-04$&$2.96$E$-04$&$2.80$E$-04$&$2.77$E$-04$\\
\hline
\end{tabular}
\end{table}

\begin{table}[ht]
\centering
\caption{Relative $L^2$-errors for Example 4 (surface $\Gamma_2$, source at $y=(0.5,0.2)^\top$, $k=10$).}\label{surf2k10}
\begin{tabular}
{|p{1.8cm}<{\centering}||p{2cm}<{\centering}|p{2cm}<{\centering}
 |p{2cm}<{\centering}|p{2cm}<{\centering}|p{2cm}<{\centering}|}
\hline
  & $h=0.08$ & $h=0.04$ & $h=0.02$ & $h=0.01$\\
\hline
\hline
$N=20$&$1.22$E$-01$&$3.80$E$-02$&$2.77$E$-02$&$2.79$E$-02$\\
\hline
$N=40$&$1.22$E$-01$&$3.24$E$-02$&$1.21$E$-02$&$1.00$E$-02$\\
\hline
$N=80$&$1.22$E$-01$&$3.20$E$-02$&$8.61$E$-03$&$3.94$E$-03$\\
\hline
$N=160$&$1.22$E$-01$&$3.21$E$-02$&$8.16$E$-03$&$2.32$E$-03$\\
\hline
$N=320$&$1.23$E$-01$&$3.22$E$-02$&$8.13$E$-03$&$2.08$E$-03$\\
\hline
\end{tabular}
\end{table}

\begin{table}[ht]
\centering
\caption{Relative $L^2$-errors for Example 5 (surface $\Gamma_3$, source at $y=(-1,0.4)^\top$, $k=1$).}\label{surf3k1}
\begin{tabular}
{|p{1.8cm}<{\centering}||p{2cm}<{\centering}|p{2cm}<{\centering}
 |p{2cm}<{\centering}|p{2cm}<{\centering}|p{2cm}<{\centering}|}
\hline
  & $h=0.16$ & $h=0.08$ & $h=0.04$ & $h=0.02$ & $h=0.01$\\
\hline
\hline
$N=20$&$1.94$E$-02$&$1.93$E$-02$&$1.92$E$-02$&$1.92$E$-02$&$1.92$E$-02$\\
\hline
$N=40$&$6.96$E$-03$&$6.84$E$-03$&$6.81$E$-03$&$6.80$E$-03$&$6.80$E$-03$\\
\hline
$N=80$&$2.61$E$-03$&$2.45$E$-03$&$2.42$E$-03$&$2.40$E$-03$&$2.40$E$-03$\\
\hline
$N=160$&$1.14$E$-03$&$9.09$E$-04$&$8.63$E$-04$&$8.51$E$-04$&$8.50$E$-04$\\
\hline
$N=320$&$7.11$E$-04$&$3.80$E$-04$&$3.16$E$-04$&$3.03$E$-04$&$3.01$E$-04$\\
\hline
\end{tabular}
\end{table}

From Tables \ref{surf1k1}-\ref{surf3k10}, we note that the indicated errors decrease in $N$ and $h$ individually up to error stagnation: For $N$ sufficiently large, the error from the discrete inverse Bloch transform is much smaller compared to the error from the finite element method, see Tables \ref{surf1k10}, \ref{surf2k10}, \ref{surf3k10} for $N=320$. Fixing $N=320$, the relative errors then decrease with the rate of $O(h^2)$ as predicted in theory, see~Theorem~\ref{th:main}. When $h$ is sufficiently small, the error from the finite element method is much smaller compared to the error from the discrete inverse Bloch transform, see Tables \ref{surf1k1}, \ref{surf2k1}, \ref{surf3k1} for $h=0.01$. Fixing $h=0.01$, the relative errors then decrease faster than the rate $O(N^{-1})$ predicted by Theorem \ref{th:main}. 
This might be seen as an indicator that the regularity result in Theorem~\ref{th:exiSolScalPerio}(b) can be somewhat improved which respect to smoothness in $\alpha$. 

\begin{table}[ht]
\centering
\caption{Relative $L^2$-errors for Example 6 (surface $\Gamma_3$, source at $y=(0.5,0.2)^\top$, $k=10$).}\label{surf3k10}
\begin{tabular}
{|p{1.8cm}<{\centering}||p{2cm}<{\centering}|p{2cm}<{\centering}
 |p{2cm}<{\centering}|p{2cm}<{\centering}|p{2cm}<{\centering}|}
\hline
  & $h=0.08$ & $h=0.04$ & $h=0.02$ & $h=0.01$\\
\hline
\hline
$N=20$&$1.10$E$-01$&$9.77$E$-02$&$1.16$E$-01$&$1.22$E$-01$\\
\hline
$N=40$&$1.14$E$-01$&$3.81$E$-02$&$4.01$E$-02$&$4.30$E$-02$\\
\hline
$N=80$&$1.21$E$-01$&$2.94$E$-02$&$1.43$E$-02$&$1.49$E$-03$\\
\hline
$N=160$&$1.25$E$-01$&$3.07$E$-02$&$8.06$E$-03$&$5.17$E$-03$\\
\hline
$N=320$&$1.26$E$-01$&$3.17$E$-02$&$7.75$E$-03$&$2.31$E$-03$\\
\hline
\end{tabular}
\end{table}

Finally, we balance the two error terms in the $L^2$-estimate $\|u_{N,h}-u\|_{L^2(\Omega_H^\Lambda)}\leq C R(u^i)  \left[h^2 + N^{-1}  \right]$ by choosing $h= c_0 N^{-1/2}$ for $c_0= 2/(5\sqrt{5})$ and $N$ equal to 20, 40, 80, and 320. This yields three pairs of $(h,N)$ equal to $(0.04,20)$, $(0.02,80)$ and $(0.01,320)$. Figure~\ref{error} shows plots in logarithmic scale of the relative $L^2$-errors for the six examples corresponding to these pairs. The lines for Example 1,3 and 5 all have slopes of about $-1.5$ while the lines of Example 2 and 4 have similar slopes of about $-1$, the line of Example 6 has a slope of $-1.36$. This shows that the numerical results converges at the rate of $N^{-1}$ or even faster.  This means that the error is bounded by $O(N^{-1})$, i.e. $O(h^2)$, which verifies the result in Theorem \ref{th:main}.

\begin{figure}[t]
\centering
\includegraphics[width=0.5\textwidth]{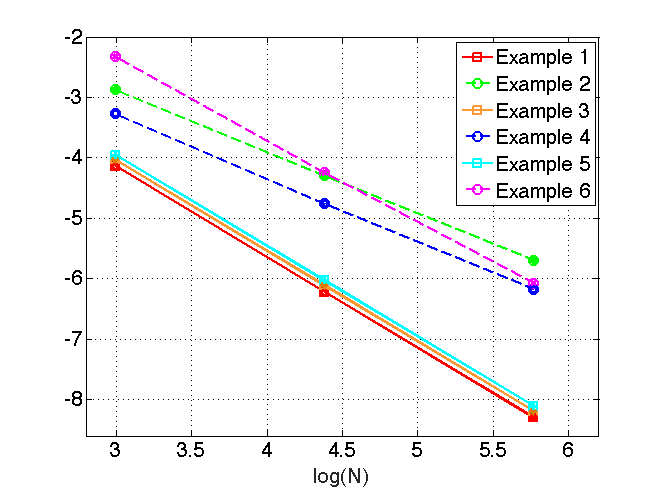}
\caption{The relative $L^2$-errors for the six considered examples with $h= c_0 N^{-1/2}$ plotted in logarithmic scale over $N$.}
\label{error}
\end{figure}

%\bibliographystyle{alpha}
%\bibliography{ip-biblio.bib} % ../../ip-biblio/ip-biblio.bib,

\providecommand{\noopsort}[1]{}

\end{document}